\documentclass[12pt,reqno]{amsart}

\usepackage{amsmath,amssymb,amsthm,amsfonts}
\usepackage{hyperref}
\usepackage{url}
\usepackage{xcolor}
\usepackage[all]{xy}
\usepackage{enumerate}

\theoremstyle{definition}
\newtheorem{definition}{Definition}[section]
\newtheorem{notation}[definition]{Notation}

\theoremstyle{plain}
\newtheorem{theorem}[definition]{Theorem}

\newtheorem{corollary}[definition]{Corollary}
\newtheorem{lemma}[definition]{Lemma}

\newtheorem{conjecture}[definition]{Conjecture}

\theoremstyle{remark}

\newtheorem{example}[definition]{Example}

\DeclareMathOperator{\Aut}{Aut}
\DeclareMathOperator{\Inn}{Inn}
\DeclareMathOperator{\Sym}{Sym}

\DeclareMathOperator{\GL}{GL}

\title{Some new compatible groups}

\author[Ding]{Zhaochen Ding}
\author[Verret]{Gabriel Verret}
\address{Department of Mathematics, University of Auckland, New Zealand}
\email{dzha470@aucklanduni.ac.nz,g.verret@auckland.ac.nz}


\begin{document}

\begin{abstract}
   Two finite groups $L_1$ and $L_2$ are compatible if there exists a finite group $G$  with isomorphic normal subgroups $N_1$ and $N_2$ such that $L_1\cong G/N_1$ and $L_2\cong G/N_2$.
   We prove a new sufficient condition for two groups to be compatible. As a corollary, we obtain that nilpotent groups of the same order are compatible, and so are groups of the same square-free order.
\end{abstract}

\maketitle

\section{Introduction}

A naive but very common mistake from students first studying group theory is to assume that if a group $G$ has two isomorphic normal subgroups $N_1$ and $N_2$, then the corresponding quotients $L_1:=G/N_1$ and $L_2:= G/N_2$ are isomorphic. There are plenty of examples to show this is not the case. The natural follow-up question is whether anything can be said about the pair $(L_1,L_2)$ in this situation. It is not hard to see that if $G$ is allowed to be infinite, then there are no restrictions on $(L_1,L_2)$. On the other hand, if one restricts to the case when $G$ is finite (and thus so are $N_1$, $N_2$, $L_1$ and $L_2$), then the question becomes interesting again and is regularly raised (see for example \cite{groupoverflow,groupexchange}).

In the rest of the paper, all groups are assumed to be finite. Two groups $L_1$ and $L_2$ are \emph{compatible} if there exists a  group $G$ with isomorphic normal subgroups $N_1$ and $N_2$ such that $L_1\cong G/N_1$ and $L_2\cong G/N_2$. Despite the conceptual simplicity of this definition, determining the compatibility of groups is surprisingly challenging. More formally, there is no known algorithm which, given as input two groups, outputs their compatibility. We call this the \emph{compatibility problem} (for abstract groups). (The smallest currently open case is whether the alternating group $\operatorname{A}_4$ and the cyclic group $\operatorname{C}_{12}$ are compatible~\cite[Problem 20.21]{khukhro2024unsolvedproblemsgrouptheory}.)

The compatibility problem was originally motivated by the study of subconstituents (\cite[Chapter III]{wielandt1964finite}) of transitive permutation groups.
Call two transitive permutation groups \emph{compatible} if they arise as paired subconstituents of a transitive group. Determining which transitive groups are compatible is the \emph{compatibility problem} (for permutation groups). Contributions to this problem include \cite{cameron1972Permutation,giudici2019Arctransitive, knapp1973point, quirin1971Primitive, sims1967Graphs}.   It turns out that two permutation groups are compatible if and only if they arise as the in- and out-local action of a finite vertex-transitive digraph with respect to some arc-transitive subgroup of automorphisms (see \cite[Theorem~3.1]{giudici2019Arctransitive}), which is how we initially became interested in the problem.

Two regular permutation groups are compatible if and only if they are compatible as abstract groups
(see \cite[Corollary~3.5]{giudici2019Arctransitive}), hence the compatibility problem for abstract groups is an instance of the compatibility problem for permutation groups.  In this paper, we will focus on the former. Compatible groups clearly must have some properties in common, for example they must have the same order and the same multiset of composition factors. Less obviously, compatible groups must have \emph{compatible subnormal series} \cite[Corollary 4.4]{ding2025necessaryconditionscompatibilitygroups}, that is subnormal series whose factors are the same and in the same order (see \cite[Definition 4.3]{ding2025necessaryconditionscompatibilitygroups} for a more formal definition). It was also shown in \cite{ding2025necessaryconditionscompatibilitygroups} that if two groups that have no abelian composition factors (such groups are sometimes called anabelian groups)
are compatible, then they must have \emph{compatible normal series}~\cite[Theorem 5.6]{ding2025necessaryconditionscompatibilitygroups}, defined in an analogous way as subnormal ones.

Our work in this paper goes in the other direction. Rather than find necessary conditions satisfied by pairs of compatible groups, we construct new examples of such pairs. Our main result, Corollary~\ref{cor:compnormalseries}, is somewhat technical but it has some straightforward consequences, for example that all nilpotent groups of the same order are compatible with each other (Corollary~\ref{cor:nilpotent}) and so are groups of the same square-free order (Corollary~\ref{cor:square-free}).

Even our most general result is only able to handle groups which have compatible normal series. Based on this and the results of \cite{ding2025necessaryconditionscompatibilitygroups}, we make the following ambitious conjecture which would completely solve the compatibility problem for abstract groups.

\begin{conjecture}
Two  groups are compatible if and only if they have compatible normal series.
\end{conjecture}

This paper is organised as follows. 
In Section~\ref{sec:prelim}, we give some preliminary results on posets, inverse systems and wreath products. 
While many of these results are well-established, we include proofs for those that are not standard. In Section~\ref{sec:inverse}, we study inverse systems of finite groups over finite in-forest posets.
%, which will be crucial in Section~\ref{sec:mainresult}. 
In Section~\ref{sec:hybridwreath}, we introduce the concept of the hybrid wreath product, a generalisation of the wreath product. 
%This section characterises the base subgroup of a hybrid wreath product using inverse systems of finite groups.
These results will then be used in Section~\ref{sec:mainresult} to prove our main results. Because these proofs our somewhat technical, we give a brief sketch of the ideas involved in
Section~\ref{sec:appendix}. Finally, Section~\ref{sec:mainresult} contains the proof of our main result, Corollary~\ref{cor:compnormalseries}, which proves that two groups admitting compatible normal series with additional properties must be compatible. As a consequence, we obtain that two nilpotent groups of the same order are compatible (Corollary~\ref{cor:nilpotent}), and two groups of the same square-free order are compatible 
(Corollaries~\ref{cor:square-free}). We also provide an example (Example~\ref{ex:counterexample}) to show that Corollary~\ref{cor:compnormalseries} cannot be applied to all groups that have compatible normal series.

\section{Preliminaries}
\label{sec:prelim}

\subsection{Forest posets}

We refer to \cite{schroder2016Ordered} for standard terms in poset theory.
Let $(I,\le)$ be a poset and let $i,j\in I$.
If $\{x\in I\mid i\le x\le j\}=\{i,j\}$, then $i$ is a \emph{lower cover} of $j$.
We denote by $i\wedge j$ the greatest element of $\{x\in I\mid x\le i,x\le j\}$, if it exists.
\begin{definition}
    A \emph{graded poset} is a poset $(I,\le)$ for which there exists a function $\rho:I\rightarrow\mathbb{N}$
    such that
    \begin{itemize}
        \item [\rm (i)] $\rho(i)\le\rho(j)$ if $i\le j$; and,
        \item [\rm (ii)] $\rho(i)+1=\rho(j)$ if $i$ is a lower cover of $j$.
    \end{itemize}
    The function $\rho$ is called a \emph{rank function} of $(I,\le)$.
\end{definition}

\begin{definition}
    Let $(I,\le)$ be a poset.
    For $i\in I$, we denote the set $\{x\in I\mid x\le i\}$ as $\downarrow i$.
    We say that $I$ is an \emph{in-forest poset} if
    for every $i\in I$, we have that $\downarrow i$ is a chain.
\end{definition}

\begin{lemma}\label{lem:partition}
    Let $(I,\le)$ be a finite in-forest poset and
    let $\rho:I\rightarrow \mathbb{N}$ be the function given by
    $\rho(i):=|\downarrow i\ |-1$.
    Then $\rho$ is a rank function of $(I,\le)$.
    In particular, $(I,\le)$ is graded.
\end{lemma}

\begin{proof}
    Let $i,j\in I$ be such that $i\le j$.
    Then $\downarrow i\subseteq \downarrow j$.
    Therefore, $\rho(i)\le \rho(j)$.
    Now suppose further that $i$ is a lower cover of $j$.
    Since both $\downarrow i$ and $\downarrow j$ are chains,
    we have that $\downarrow j=\downarrow i\ \cup\{j\}$.
    Therefore, $\rho(j)=\rho(i)+1$.
    This completes the proof.
\end{proof}

\begin{lemma}\label{lem:existenceofmeet}
    Let $(I,\le)$ be a finite in-forest poset and 
    let $i, j,k\in I$ such that $i\le j$.
    We have that
    $j\wedge k$ exists if and only if $i\wedge k$ exists.
        Moreover, in that case, either $i\wedge k=i$ or $i\wedge k=j\wedge k$.
\end{lemma}

\begin{proof}
    Suppose that $j\wedge k$ exists. 
    Then $j\wedge k\in \downarrow j$ and $i\in \downarrow j$.
    As $\downarrow j$ is a chain, either $i\le j\wedge k$ or $i\ge j\wedge k$.
    If $i\le j\wedge k$, then $i\le k$ which implies that $i\wedge k$ exists and $i\wedge k=i$.
    If $i\ge j\wedge k$, then $j\wedge k\in\downarrow i\ \cap\downarrow k$.
    Therefore, $\downarrow i\ \cap\downarrow k$ is a 
    non-empty finite chain, so it has a greatest element $i\wedge k$. 
    As $i\le j$, we have $i\wedge k\le j\wedge k$,
    but $i\ge j\wedge k$ also implies that $i\wedge k\ge j\wedge k$.
    Hence $i\wedge k=j\wedge k$ in this case.
    In the other direction, suppose that $i\wedge k$ exists.
    Then $i\wedge k\in \downarrow j\ \cap\downarrow k$.
    Therefore, $\downarrow j\ \cap\downarrow k$ is a non-empty finite chain, which implies that $j\wedge k$ exists. This completes the proof.
\end{proof}

\subsection{Inverse systems}

Throughout this paper, we denote the category of sets as $\mathbf{Set}$,
the category of groups as $\mathbf{Grp}$, and the category of finite groups as $\mathbf{FinGrp}$.
When we discuss categories, we always mean one of these three categories.

If $(X,(p_i)_{i\in I})$ is an inverse limit of the inverse system $\mathcal{X}=((X_i)_{i\in I}, (f_{ij})_{i\le j\in I})$,
we will write $\varprojlim_{i\in I}(\mathcal{X})=(X,(p_i)_{i\in I})$.
For convenience, we often omit the poset if it is clear from context, for example, we might write $((X_i),(f_{ij}))$ instead of $((X_i)_{i\in I}, (f_{ij})_{i\le j\in I})$, and
 $\varprojlim(\mathcal{X})=(X,(p_i))$ or even $\varprojlim(\mathcal{X})=X$ instead of $\varprojlim_{i\in I}(\mathcal{X})\\=(X,(p_i)_{i\in I})$.

By the following theorem, inverse limits of inverse systems in $\mathbf{Grp}$ or $\mathbf{Set}$ always exist.
This theorem is standard.

\begin{theorem}[{\cite[Example 5.1.22 and 5.1.23]{Leinster_2014}}]\label{thm:limitexists}
    Let $(I,\le)$ be a poset and let $\mathcal{X}:=((X_i),(f_{ij}))$ be an inverse system in $\mathbf{Grp}$ or $\mathbf{Set}$ over $I$.
    Let 
    \[X:=\left\{(x_i) \in \prod_{i \in I} X_{i} \mid x_{i}=f_{i j}(x_j) \text { for all } i \leq j \text { in } I\right\}\]
    and for every $i\in I$, let 
    $p_i:X\rightarrow X_i$ be the natural projection (given by $p_i((x_i))=x_i$).
    Then $(X,(p_i))$ is a inverse limit of $\mathcal{X}$.
\end{theorem}

The following corollary is a consequence of Theorem \ref{thm:limitexists},
which plays a fundamental role in our analysis.

\begin{corollary}\label{cor:fininverse limitexists}
    The inverse limit of an inverse system in $\mathbf{FinGrp}$ over a finite poset exists.
\end{corollary}
\begin{proof}
    Let $(I,\le)$ be a finite poset and let $\mathcal{X}:=((X_i),(f_{ij}))$ be an inverse system in $\mathbf{FinGrp}$.
    Let \[X:=\left\{(x_i) \in \prod_{i \in I} X_{i} \mid x_{i}=f_{i j}(x_j) \text { for all } i \leq j \text { in } I\right\}\]
    equipped with natural projection $p_i:X\rightarrow X_i$
    for every $i\in I$.
    Since $I$ is finite and $X_i$ is also finite for every $i\in I$,
    we have that $X$ is a finite group.
    By Theorem \ref{thm:limitexists}, $(X,(p_i))$ is an inverse limit of $\mathcal{X}$ in $\mathbf{Grp}$.
    Therefore, it is also an inverse limit of $\mathcal{X}$ in $\mathbf{FinGrp}$.
\end{proof}

\begin{lemma}[{\cite[Proposition 6.1.4]{Leinster_2014}}]\label{lem:limiso}
    Let $(I,\le)$ be a poset, let $\mathbf{C}$ be a category,
    and let $\mathcal{X}$ and $\mathcal{Y}$ be inverse systems in $\mathbf{C}$ over $I$.
    Let $\Phi:=(\phi_i):\mathcal{X}\rightarrow\mathcal{Y}$ be morphisms. 
    Assume inverse limits of $\mathcal{X}$ and $\mathcal{Y}$ exist. 
    If $\phi_i$ is an isomorphism for every $i\in I$, then $\varprojlim(\Phi)$ is an isomorphism.
\end{lemma}

\subsection{Wreath products}

\begin{notation}
    If $G$ is a group, 
    then $\Inn_G:G\rightarrow \Aut(G)$ is 
    defined as $\Inn_G(g)(h):=ghg^{-1}$ for
    every $g,h\in G$.
    We omit the subscript of $\Inn_G$
    if there is no ambiguity.
\end{notation}

The standard form of the Krasner-Kaloujnine universal embedding theorem can be found in \cite[Theorem 2.6A]{dixon1996Permutation}.
In this paper, we give a slightly more general version of the universal embedding theorem, see Theorem \ref{thm:uet}.
Before that, we set some notation and definitions.

    Let $\Omega$ be a set on which permutations act on the right.
    Given a (right) group action 
    $\rho:G\rightarrow \operatorname{Sym}(\Omega)$, $g\in G$ and $\omega\in\Omega$,
    we denote the action of $\rho(g)$ on $\omega$ by $\omega^{\rho(g)}$, i.e.,
    \[\omega^{\rho(g)}:=(\omega)\rho(g).\]
    We usually omit $\rho$ if there is no ambiguity,
    that is, 
    we might write $\omega^g$ instead of $\omega^{\rho(g)}$ or
    write `let $G$ act on $\Omega$' without giving a name to this action.

    Let $G$ and $H$ be groups and let $H$ act on $\Omega$.
    Define an action of $H$ on $G^\Omega$
    as:
    \[f^h(\omega):=f(\omega^{h^{-1}})\]
    for every $f\in G^\Omega$, $h\in H$ and $\omega\in\Omega$.
    The semi-direct product $G^\Omega\rtimes H$ with respect to
    this action is called the \emph{wreath product} of $G$ by $H$,
    denoted by $G\wr_\Omega H$.
    The subgroup $G^\Omega\rtimes 1$ is called the \emph{base subgroup} of this wreath product.
    We often identify this base subgroup with $G^\Omega$.

    \begin{definition}\label{def:universalembedding}
    Let $\rho:G\rightarrow \operatorname{Sym}(\Omega)$ be a transitive group action of $G$ on $\Omega$.
    Fix $\omega\in\Omega$. 
    For every $\nu\in\Omega$, let $t_\nu$ be an element of $G$ such that
    $\nu=\omega^{t_\nu}$. Without loss of generality, we set $t_\omega=1$.
    We call $(t_\nu)_{\nu\in\Omega}$ a \emph{permutation transversal} with respect to $\omega$.
    Let $H:=G_\omega$, let $S:=\rho(G)$ and
    for $g\in G$ define $f_g\in H^\Omega$ via $f_g(\nu):=t_\nu gt^{-1}_{\nu^g}$ for every $\nu\in\Omega$.
    Define a map 
    \[\begin{aligned}
        \iota:G&\rightarrow H\wr_\Omega S\\
        g&\mapsto (f_g,\rho(g)).
    \end{aligned}\]
    The map $\iota$ is called the \emph{standard embedding} of $G$ with respect to $\omega$ and $(t_\nu)_{\nu\in\Omega}$.
    \end{definition}

\begin{theorem}[Universal embedding theorem]\label{thm:uet}
    A standard embedding is an injective homomorphism.
\end{theorem}
One can easily prove this theorem by adapting the the proof of the usual universal embedding theorem.

Now we discuss the uniqueness of the standard embedding and the wreath product of homomorphisms.

\begin{lemma}\label{lem:uniqueembedding}
    Let $\rho:G\rightarrow \operatorname{Sym}(\Omega)$ be a transitive group action of $G$ on $\Omega$.
    Fix $\omega\in\Omega$. 
    Let $(t_\nu)_{\nu\in\Omega}$ and $(s_\nu)_{\nu\in\Omega}$ be permutation transversals with respect to $\omega$.
    Let $\iota$ be the standard embedding of $G$ with respect to $\omega$ and $(t_\nu)_{\nu\in\Omega}$,
    and let $\lambda$ be the standard embedding of $G$ with respect to 
    $\omega$ and $(s_\nu)_{\nu\in\Omega}$.
    There exists $f\in H^\Omega$ such that $\lambda=\Inn_{H\wr_\Omega \rho(G)}(f)\circ\iota$.
\end{lemma}

\begin{proof}
    Let $f\in H^\Omega$ defined as $f(\nu):=s_\nu t_\nu^{-1}$.
    Note that $s_\nu t_\nu^{-1}\in H$ for every $\nu\in\Omega$ so $f$ is well-defined.
    Let $g\in G$.
    We show that $\Inn_{H\wr_\Omega \rho(G)}(f)\circ\iota(g)=\lambda(g)$.
    Write $\iota(g)=(f_g,\rho(g))$ and $\lambda(g)=(h_g,\rho(g))$.
    We have
    \[(\Inn_{H\wr_\Omega \rho(G)}(f)\circ \iota)(g)=(f,1)(f_g,\rho(g))(f^{-1},1)=(f f_g(f^{-1})^{\rho(g)^{-1}},\rho(g)).\]
    For every $\nu\in\Omega$,
    we have \[\begin{aligned}
        f f_g(f^{-1})^{\rho(g)^{-1}}(\nu)&=(s_\nu t_\nu^{-1})(t_\nu g t_{\nu^g}^{-1})(t_{\nu^g}s_{\nu^g}^{-1})\\
        &=s_\nu g s_{\nu^g}^{-1}\\
        &=h_g(\nu)
    \end{aligned}\]
    It follows that $f f_g(f^{-1})^{\rho(g)^{-1}}=h_g$.
    Therefore, $\lambda=\Inn_{H\wr_\Omega \rho(G)}(f)\circ\iota$.
\end{proof}

Lemma \ref{lem:uniqueembedding} indicates that the standard embedding of a group $G$ 
with respect to a given point is unique up to conjugation.

Now we introduce the wreath product of homomorphisms.
    Let $H$ and $E$ be groups, let $\rho:H\rightarrow\Sym(\Omega)$
    be a group action of $H$ on $\Omega$, 
    and let $\kappa:E\rightarrow\Sym(\Gamma)$
    be a group action of $E$ on $\Gamma$.
    If there exist a bijection $\phi:\Omega\rightarrow \Gamma$ and an isomorphism $\psi:H\rightarrow E$
    such that for every $\omega\in\Omega$ and every $h\in H$, 
    we have $\phi(\omega^{\rho(h)})=\phi(\omega)^{\kappa(\psi(h))}$, then
    the actions $\rho$ and $\kappa$ are \emph{isomorphic} and
    $(\phi,\psi)$ is an \emph{action isomorphism} from $\rho$ to $\kappa$.

\begin{definition}\label{def:wrofhomo}
    Let $H$, $E$, $G$, and $K$ be groups,
    let $\rho:H\rightarrow\Sym(\Omega)$
    be a group action of $H$ on $\Omega$, and 
    let $\kappa:E\rightarrow\Sym(\Gamma)$
    be a group action of $E$ on $\Gamma$.
    Let $(\phi:\Omega\rightarrow\Gamma,\psi:H\rightarrow E)$ 
    be an action isomorphism from $\rho$ to $\kappa$.
    Let $\eta:G\rightarrow K$ be a homomorphism.
    Define a map 
    \[\begin{aligned}
        \eta\wr_\phi\psi:G\wr_\Omega H&\rightarrow K\wr_{\Gamma}E\\
        (f,h)&\mapsto (\eta\circ f\circ\phi^{-1},\psi(h)).
    \end{aligned}
    \]
    We call $\eta\wr_\phi\psi$ the \emph{wreath product} of $\eta$ by $(\phi,\psi)$.
\end{definition}

\begin{lemma}\label{lem:wrhomo}
    Following the notation in Definition \ref{def:wrofhomo},
    we have that
    \begin{itemize}
        \item [\rm (i)] $\eta\wr_\phi\psi$ is a homomorphism, and
        \item [\rm (ii)] $\eta\wr_\phi\psi$ is surjective
        if and only if $\eta$ is surjective.
    \end{itemize}
\end{lemma}
\begin{proof}
    {\rm (i)} Let $(f_1,h_1),(f_2,h_2)\in G\wr_\Omega H$.
    We have 
    \[\begin{aligned}
        (\eta\wr_\phi\psi(f_1,h_1))(\eta\wr_\phi\psi(f_2,h_2))&=(\eta\circ f_1\circ\phi^{-1},\psi(h_1)) (\eta\circ f_2\circ\phi^{-1},\psi(h_2))\\
        % &=((\eta\circ f_1\circ\phi^{-1})(\eta\circ f_2\circ\phi^{-1})^{\psi(h_1)^{-1}},\psi(h_1)\psi(h_2))\\
        &=((\eta\circ f_1\circ\phi^{-1})(\eta\circ f_2\circ\phi^{-1})^{\psi(h_1)^{-1}},\psi(h_1h_2))
    \end{aligned}
    \]
    Let $\gamma\in\Gamma$.
    Since $(\phi,\psi)$ is an action isomorphism, we have $\phi^{-1}(\gamma^{\psi(h_1)})=\phi^{-1}(\gamma)^{h_1}$.
    Therefore,
    \[\begin{aligned}
        (\eta\circ f_2\circ \phi^{-1})^{\psi(h_1)^{-1}}(\gamma)&=(\eta\circ f_2\circ \phi^{-1})(\gamma^{\psi(h_1)})\\
        % &=(\eta\circ f_2) (\phi^{-1}(\gamma)^{h_1})\\
        % &=(\eta\circ f_2^{h_1^{-1}})(\phi^{-1}(\gamma))\\
        &=(\eta\circ f_2^{h_1^{-1}}\circ \phi^{-1})(\gamma)
    \end{aligned}\]
    This follows that $(\eta\circ f_2\circ \phi^{-1})^{\psi(h_1)^{-1}}=\eta\circ f_2^{h_1^{-1}}\circ \phi^{-1}$.
    Also note that 
    \[\begin{aligned}
        (\eta\circ f_1\circ\phi^{-1})(\eta\circ f_2^{h_1^{-1}}\circ\phi^{-1})(\gamma)&=(\eta\circ f_1\circ\phi^{-1})(\gamma)(\eta\circ f_2^{h_1^{-1}}\circ\phi^{-1})(\gamma)\\
        % &=\eta(f_1\circ\phi^{-1}(\gamma))\eta(f_2^{h_1^{-1}}\circ \phi^{-1}(\gamma))\\
        % &=\eta((f_1\circ\phi^{-1}(\gamma))(f_2^{h_1^{-1}}\circ\phi^{-1}(\gamma)))\\
        % &=\eta((f_1(\phi^{-1}(\gamma)))f_2^{h_1^{-1}}(\phi^{-1}(\gamma)))\\
        % &=\eta((f_1f_2^{h_1^{-1}})(\phi^{-1}(\gamma)))\\
        &=(\eta\circ(f_1f_2^{h_1^{-1}})\circ\phi^{-1})(\gamma).
    \end{aligned}\]
    Hence $(\eta\circ f_1\circ\phi^{-1})(\eta\circ f_2^{h_1^{-1}}\circ\phi^{-1})=\eta\circ(f_1f_2^{h_1^{-1}})\circ\phi^{-1}$.
    Therefore, 
    \[\begin{aligned}
        (\eta\wr_\phi\psi(f_1,h_1))(\eta\wr_\phi\psi(f_2,h_2))&=(\eta\circ(f_1f_2^{h_1^{-1}})\circ\phi^{-1},\psi(h_1h_2))\\
        % &=\eta\wr_{\phi}\psi(f_1f_2^{h_1^{-1}},h_1h_2)\\
        &=\eta\wr_\phi\psi((f_1,h_1)(f_2,h_2))
    \end{aligned}
    \]
    which implies that $\eta\wr_\phi\psi$ is a homomorphism.

    {\rm (ii)}
    Fix $\gamma\in\Gamma$.
    If $\eta\wr_\phi\psi$ is surjective, then let $k\in K$, let 
    $\delta_k\in K^\Gamma$ be given by $\delta_k(\gamma):=k$ and 
    $\delta_k(\alpha):=1$ for every $\alpha\in \Gamma\setminus\{\gamma\}$. 
    Then $(\delta_k,1)\in K\wr_{\Gamma}E$.
    Since $\eta\wr_\phi\psi$ is surjective, there exists $(f,h)\in G\wr_{\Omega} H$
    such that $\eta\wr_\phi\psi(f,h)=(\delta_k,1)$.
    This implies $(\eta\circ f\circ\phi^{-1},\psi(g))=(\delta_k,1)$.
    Hence $\eta(f\circ\phi^{-1}(\omega))=\delta_k(\omega)=k$.
    This follows that $\eta$ is surjective.

    We now consider the case when $\eta$ is surjective.
    Let $(f,e)\in K\wr_{\Gamma}E$.
    Since $\eta$ is surjective, we have that for every $\alpha\in\Gamma$,
    there exists $g_\alpha\in G$ such that $\eta(g_\alpha)=f(\alpha)$.
    Let $f_1\in G^\Omega$ where $f_1(\nu):=g_{\phi(\nu)}$ for every $\nu\in\Omega$.
    We then have 
    $(\eta\circ f_1\circ\phi^{-1})(\alpha)=\eta(g_\alpha)=f(\alpha)$ for every $\alpha\in\Gamma$.
    Therefore $\eta\circ f_1\circ\phi^{-1}=f$ and
    \[\eta\wr_\phi\psi(f_1,\psi^{-1}(e))=(\eta\circ f_1\circ\phi^{-1},e)=(f,e).\]
    This implies that $\eta\wr_\phi\psi$ is surjective.
\end{proof}

\section{Inverse systems of finite groups}
\label{sec:inverse}

For the remainder of this paper, 
we will always assume that our inverse systems
are in $\mathbf{FinGrp}$ unless explicitly stated.
We also assume that all posets are finite.
Then by Corollary~\ref{cor:fininverse limitexists}, the inverse limit of the inverse systems in $\mathbf{FinGrp}$ always exists.
We often omit the partial order symbol
when mentioning posets, for example, we might write $I$ rather than $(I,\le)$ if the order is clear from the context.

\begin{notation}
    Let $f:X\rightarrow Y$ be a map, let $A\subseteq X$ and let $B\subseteq Y$ such that $f(A)\subseteq B$. 
    We define the \emph{restriction} $f|_{A,B}:A\rightarrow B,x\mapsto f(x)$.
\end{notation}

\begin{definition}
    Let $\mathcal{X}:=((X_i),(f_{ij}))$ be an inverse system over $I$.
    If $Y_i$ is a subgroup of $X_i$ for every $i\in I$ and
    $f_{ij}(Y_j)\le Y_i$ for every $i\le j\in I$, then
    $\mathcal{Y}:=((Y_i),(f_{ij}|_{Y_j,Y_i}))$ is also an inverse system over $I$.
    We call $\mathcal{Y}$ a \emph{subsystem} of $\mathcal{X}$.    
\end{definition}

A reader with some background knowledge of category theory or profinite group theory may find
Lemma~\ref{lem:subinverse limituniqueness} and Theorem~\ref{thm:limpullback} familiar.
However, as we have not found a direct reference, we include the proofs for completeness.

\begin{lemma}\label{lem:subinverse limituniqueness}
    Let $\mathcal{X}:=((X_i),(f_{ij}))$ be an inverse system over $I$
    and let \[\mathcal{Y}:=((Y_i),(f_{ij}|_{Y_j,Y_i}))\]
    be a subsystem of $\mathcal{X}$.
    If $(X,(p_i))=\varprojlim(\mathcal{X})$, 
    then there exists a subgroup $Y$
    of $X$ such that $p_i(Y)\le Y_i$ for every $i\in I$ and $(Y,(p_i|_{Y,Y_i}))=\varprojlim(\mathcal{Y})$.
    Moreover, if $Z$ is a subgroup of $X$ such that $p_i(Z)\le Y_i$ for every $i\in I$, then $Z\le Y$.
\end{lemma}
\begin{proof}
    First, we show existence. By Theorem~\ref{thm:limitexists}, we can write
    \[X{\color{teal}=}\left\{(x_i) \in \prod_{i \in I} X_{i} \mid x_{i}=f_{i j}(x_j) \text { for all } i \leq j \text { in } I\right\}\]
    equipped with the natural projections $p_i:X\rightarrow X_i$.
    Let
    \[Y:=\left\{(y_i) \in \prod_{i \in I} Y_{i} \mid y_{i}=f_{i j}(y_j) \text { for all } i \leq j \text { in } I\right\}\]
    equipped with the natural projections $q_i:Y\rightarrow Y_i$.
    By Theorem~\ref{thm:limitexists}, $(Y,(q_i))$ is an inverse limit of $\mathcal{Y}$.
    By definition, $Y_i\le X_i$ for each $i$. 
    It can be seen that $\prod_i Y_i\le\prod_i X_i$, $Y\le X$
    and $q_i=p_i|_{Y,Y_i}$. This completes the existence part.

    We now show the second part of our statement. 
    Let $Z\le X$ such that $p_i(Z)\le Y_i$ for every $i\in I$.
    Note that $f_{ij}|_{Y_j,Y_i}\circ p_j|_{Z,Y_j}=p_i|_{Z,Y_i}$ for every $i\le j\in I$.
    Then by the definition of inverse limits, 
    there exists a unique morphism $u:Z\rightarrow Y$ such that
    $p_i|_{Z,Y_i}=p_i|_{Y,Y_i}\circ u$.
    Since $Y_i\le X_i$, we also have $p_i|_{Y,X_i}\circ u=p_i|_{Z,X_i}$,
    $f_{ij}\circ p_j|_{Y,X_j}=p_i|_{Y,X_i}$ and $f_{ij}\circ p_j|_{Z,X_j}=p_i|_{Z,X_i}$
    for  every $i\le j\in I$.
    Let $\iota_Y:Y\rightarrow X$ and $\iota_Z:Z\rightarrow X$ be inclusions.
    Note that $p_i|_{Y,X_i}=p_i\circ\iota_Y$. We have that 
    \[p_i|_{Z,X_i}=p_i|_{Y,Y_i}\circ u=p_i\circ\iota_Y\circ u.\]
    Thus we have a diagram as shown in Figure \ref{fig:XYZ}.
    Also note that
    $p_i|_{Z,X_i}=p_i\circ\iota_Z$ and recall that $X$ is the inverse limit of $\mathcal{X}$.
    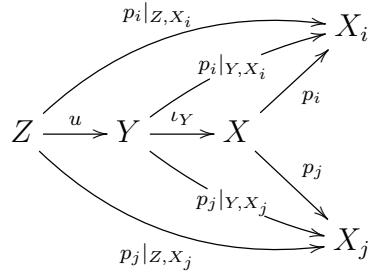
\begin{figure}[ht]
    \[\xymatrix{
        &&&X_i\\
        Z\ar[r]^{u}\ar@/^1.5pc/[rrru]^{p_i|_{Z,X_i}}\ar@/_1.5pc/[rrrd]_{p_j|_{Z,X_j}}&Y\ar[r]^{\iota_Y}\ar@/^/[rru]|{p_i|_{Y,X_i}}\ar@/_/[rrd]|{p_j|_{Y,X_j}}&X\ar[ru]_{p_i}\ar[rd]^{p_j}\\
        &&&X_j
    }\]
    \caption{Diagram for the proof of Lemma~\ref{lem:subinverse limituniqueness}}
    \label{fig:XYZ}
    \end{figure}
    Then by the definition of inverse limits, we have that $\iota_Z=\iota_Y\circ u$.
    This implies that $z=\iota_Z(z)=\iota_Y(u(z))=u(z)$ for every $z\in Z$.
    Hence $u$ is the inclusion from $Z$ to $Y$ which implies that $Z\le Y$.
    This completes the proof.
\end{proof}

\begin{definition}\label{def:invker}
    Let $\mathcal{X}:=((X_i),(f_{ij}))$ and $\mathcal{Y}:=((Y_i),(g_{ij}))$ be inverse systems over $I$
    and let $\Phi:=(\phi_i):\mathcal{X}\rightarrow \mathcal{Y}$ be a morphism.
    If $X_i$ is trivial for each $i$, then we say that $\mathcal{X}$ is a \emph{trivial inverse system} and
    if $\phi_i$ is trivial for each $i$, then we say that $\Phi$ is a \emph{trivial morphism}.
    Let $\mathcal{Z}:=((Z_i),(g_{ij}|_{Z_j,Z_i}))$ be a subsystem of $\mathcal{Y}$.
    Define
    \[\Phi^{-1}(\mathcal{Z}):=((\phi_i^{-1}(Z_i)),(f_{ij}|_{\phi_j^{-1}(Z_j),\phi_i^{-1}(Z_i)})),\]
    where $\phi_i^{-1}(Z_i)$ is the full preimage of $Z_i$ and define
    \[\ker(\Phi):=((\ker(\phi_i)),(f_{ij}|_{\ker(\phi_j),\ker(\phi_i)})).\]
    % and \[{\rm im}(\Phi):=(({\rm im}(\phi_i)),(g_{ij}|_{{\rm im}(\phi_j),{\rm im}(\phi_i)})).\]
\end{definition}

It can be seen that $\Phi^{-1}(\mathcal{Z})$ and $\ker(\Phi)$ are well-defined subsystems of $\mathcal{X}$.

\begin{theorem}\label{thm:limpullback}
    Assume the notation in Definition~\ref{def:invker}.
    Let $(X,(p_i)):=\varprojlim(\mathcal{X})$, let $(Y,(q_i)):=\varprojlim(\mathcal{Y})$
    and let $Z\le Y$ such that for every $i$,  we have that $q_i(Z)\le Z_i$ and $(Z,(q_i|_{Z,Z_i}))=\varprojlim(\mathcal{Z})$.
    If $\phi:=\varprojlim(\Phi):X\rightarrow Y$, then for every $i\in I$, we have that
    \begin{itemize}
        \item[\rm(i)] $p_i(\phi^{-1}(Z))\le \phi_i^{-1}(Z_i)$; 
        \item[\rm(ii)] $(\phi^{-1}(Z),p_i|_{\phi^{-1}(Z),\phi^{-1}(Z_i)})=\varprojlim(\Phi^{-1}(\mathcal{Z}))$;
        \item[\rm(iii)] $p_i(\ker(\phi))\le\ker(\phi_i)$ and $(\ker(\phi),p_i|_{\ker(\phi),\ker(\phi_i)})=\varprojlim(\ker(\Phi))$.
    \end{itemize}
\end{theorem}
\begin{proof}
    {\rm (i)} Let $i\in I$.
    Note that $q_{i}\circ\phi=\phi_i\circ p_i$.
    Then \[\phi_i(p_i(\phi^{-1}(Z)))=q_i(\phi(\phi^{-1}(Z)))\le q_i(Z)\le Z_i\]
    which implies that $p_i(\phi^{-1}(Z))\le \phi_i^{-1}(Z_i)$.
    
    {\rm (ii)}
    Let $T\le X$ such that $p_i(T)\le \phi_i^{-1}(Z_i)$ and 
    $(T,(p_i|_{T,\phi^{-1}(Z_i)}))=\varprojlim(\Phi^{-1}(\mathcal{Z}))$.
    Then by Lemma~\ref{lem:subinverse limituniqueness} and the fact
    that $p_i(\phi^{-1}(Z))\le\phi^{-1}_i(Z_i)$ for every $i\in I$, 
    we have that $\phi^{-1}(Z)\le T$.
    On the other hand, note that
    \[q_i(\phi(T))=\phi_i(p_i(T))\le\phi_i(\phi^{-1}_i(Z_i))\le Z_i.\]
    Then by Lemma~\ref{lem:subinverse limituniqueness}, 
    $\phi(T)\le Z$ (see Figure \ref{fig:pullback}), which implies that
    $T\le \phi^{-1}(Z)$.
    \begin{figure}[ht]
        \[\xymatrix{
        &&Z_i\\
        \phi(T)\ar@{^{(}->}[r]\ar@/^/[rru]^{q_i|_{\phi(T),Z_i}\ }\ar@/_/[rrd]_{q_j|_{\phi(T),Z_j}}& Z\ar[ru]|{q_i|_{Z,Z_i}}\ar[rd]|{q_j|_{Z,Z_j}}\\
        &&Z_j
    }\]
    \caption{Diagram for the proof of Theorem~\ref{thm:limpullback}}
    \label{fig:pullback}
    \end{figure}
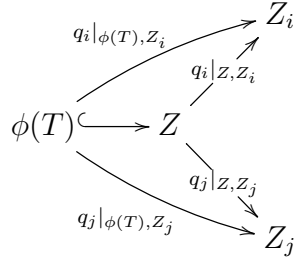
    Therefore, $T=\phi^{-1}(Z)$, which completes the proof of {\rm (ii)}.
    
    As for {\rm (iii)}, we point out that $\ker(\Phi)=\Phi^{-1}(\mathcal{Z}_0)$ where
    $\mathcal{Z}_0$ is the trivial subsystem of $\mathcal{Y}$.
    Also note that $\varprojlim(\mathcal{Z}_0)=1$.
    Hence $\ker(\Phi)$ is simply a special case of $\Phi^{-1}(\mathcal{Z})$.
\end{proof}

We can write the equation in Theorem~\ref{thm:limpullback} as
$\varprojlim(\Phi^{-1}(\mathcal{Z}))=\phi^{-1}(\varprojlim(\mathcal{Z}))$.
In other words, inverse limit and pullback commute.

\begin{lemma}\label{lem:nonemptylim}
    Let $\mathcal{X}:=((X_i),(f_{ij}))$ be a surjective inverse system in $\mathbf{Set}$ 
    over an in-forest poset $I$ and
    let $(X,(p_i))=\varprojlim(\mathcal{X})$.
    If $X_i$ is non-empty for every $i\in I$, then $X$ is non-empty.
\end{lemma}

\begin{proof}
    Let \[X:=\left\{(x_i) \in \prod_{i \in I} X_{i} \mid x_{i}=f_{i j}(x_j) \text { for all } i \leq j \text { in } I\right\}\]
    equipped with the natural projections $p_i:X\rightarrow X_i$ for every $i\in I$.
    By Theorem~\ref{thm:limitexists}, $(X,(p_i))=\varprojlim(\mathcal{X})$.
    Therefore, it is sufficient to show that $X$ is non-empty.
    By Lemma~\ref{lem:partition}, let $\rho$ be the rank function of $I$
    and let $S_k:=\{i\in I\mid \rho(i)=k\}$ for every $k\in \mathbb{N}$.
    Note that $\{S_k\mid k\in\mathbb{N}\}$ forms a partition of $I$.
    In particular, $I=\bigcup_{k\in\mathbb{N}} S_k$.

    Now we define $(x_i)_{i\in I}$ in an inductive way.
    For every $i\in S_0$, let $x_i\in X_i$.
    For $i\in S_\alpha$ where $\alpha>0$, we assume that $x_j\in X_j$ has been defined for every $j\in S_0\cup\cdots\cup S_{\alpha-1}$.
    Note that there exists $i'\in S_{\alpha-1}$ such that $i$ covers $i'$.
    By the definition of in-forest poset, this $i'$ is unique.
    Let $x_i$ be an element in $f_{i'i}^{-1}(x_{i'})$.
    Since $I=\bigcup_{k\in\mathbb{N}} S_k$, we have defined $x_i\in X_i$
    for every $i\in I$.
    We show that $(x_i)_{i\in I}\in X$.
    Let $i,j\in I$ such that $i\le j$.
    By the definition of in-forest poset,
    there exists a unique sequence
    $i=i_0, i_1,\ldots, i_n=j$ such that
    $i_k$ covers $i_{k-1}$ for $k\in\{1,\ldots,n\}$.
    Therefore, \[f_{ij}(x_j)=f_{i_0i_1}\circ f_{i_1i_2}\circ\cdots\circ f_{i_{n-1}i_n}(x_j)=x_i\]
    Hence $(x_i)\in X$ which completes the proof.
\end{proof}

Now we are back to inverse systems in $\mathbf{FinGrp}$.
\begin{theorem}\label{thm:limsurj}
    Let $I$ be an in-forest poset.
    Let $\mathcal{X}:=((X_i),(f_{ij}))$ and $\mathcal{Y}:=((Y_i),(g_{ij}))$ be 
    inverse systems over $I$
    and let $\Phi:=(\phi_i):\mathcal{X}\rightarrow \mathcal{Y}$ be a morphism.
    If $\phi_i$ is surjective for every $i\in I$ and $\ker(\Phi)$ is a surjective inverse system,
    then $\varprojlim(\Phi)$ is surjective.
\end{theorem}

\begin{proof}
    Let
    \[X:=\left\{(x_i) \in \prod_{i \in I} X_{i} \mid x_{i}=f_{i j}(x_j) \text { for all } i \leq j \text { in } I\right\}\]
    and
    \[Y:=\left\{(y_i) \in \prod_{i \in I} Y_{i} \mid y_{i}=g_{i j}(y_j) \text { for all } i \leq j \text { in } I\right\}\]
    equipped with the natural projections $p_i:X\rightarrow X_i$ and $q_i:Y\rightarrow Y_i$ for every $i\in I$ respectively.
    By Theorem \ref{thm:limitexists}, $(X,(p_i))=\varprojlim(\mathcal{X})$ and $(Y,(q_i))=\varprojlim(\mathcal{Y})$.
    Let $\phi:X\rightarrow Y$ be given by 
    $\phi((x_i)):=(\phi_i(x_i))$.
    Note that $q_i\circ\phi((x_i))=\phi_i(x_i)=\phi_i\circ p_i((x_i))$ for every $i\in I$.
    Then $\phi=\varprojlim(\Phi)$.
    Let $(y_i)\in Y$ and for every $i\in I$, let $A_i:=\phi_i^{-1}(y_i)$.
    Since $\phi_i$ is surjective, we have that $A_i$ is non-empty.
    Let $i,j\in I$ such that $i\le j$.
    Note that $\phi_i\circ f_{ij}=g_{ij}\circ\phi_j$.
    We have that \[\phi_i(f_{ij}(A_j))=g_{ij}(\phi_j(A_j))=g_{ij}(y_j)=y_i\]
    which implies that $f_{ij}(A_j)\le \phi_i^{-1}(y_i)=A_i$.
    Therefore, $((A_i),(f_{ij}|_{A_j,A_i}))$ is an inverse system in $\mathbf{Set}$ over $I$ with non-empty $A_i$ for every $i\in I$.
    Let $\mathcal{A}:=((A_i),(f_{ij}|_{A_j,A_i}))$.
    Now we show that $\mathcal{A}$ is a surjective inverse system.
    Let $i\le j\in I$,
    let $K_i:=\ker(\phi_i)$, let $K_j:=\ker(\phi_j)$ and let $a_j\in A_j$.
    Then $A_j=K_ja_j$ and $f_{ij}(A_j)=f_{ij}(K_j)f_{ij}(a_j)$.
    Since $\ker(\Phi)$ is a surjective inverse system, 
    we have that $f_{ij}(A_j)=K_if_{ij}(a_j)=A_i$
    which implies that $\mathcal{A}$ is a surjective inverse system.
    By Lemma \ref{lem:nonemptylim}, $\varprojlim(\mathcal{A})$ is non-empty.
    Let
    \[A:=\left\{(a_i) \in \prod_{i \in I} A_{i} \mid a_{i}=f_{i j}(a_j) \text { for all } i \leq j \text { in } I\right\}\]
    equipped with the natural projections $p'_i:A\rightarrow A_i$ for every $i\in I$.
    By Theorem \ref{thm:limitexists}, $(A,(p'_i))=\varprojlim(\mathcal{A})$.
    Hence $A\ne \emptyset$. Take $(a_i)\in A$.
    Note that $A\subseteq X$.
    Then $\phi((a_i))=(\phi_i(a_i))=(y_i)$.
    Therefore, $\phi$ is surjective which completes the proof.
\end{proof}

Now we characterise the inverse limit projections of a surjective inverse system over an in-forest poset
in Theorem \ref{thm:projection}.
\begin{definition}\label{def:limprojection}
    Let $I$ be an in-forest poset.
    Let $\mathcal{X}:=((X_i),(f_{ij}))$ be an inverse system over $I$.
    Given $i_0\in I$, for every $i\in I$
    let $Y_i:=X_{i\wedge i_0}$ if $i\wedge i_0$ exists, 
    and let $Y_i:=1$ otherwise. 
    For every $i\le j\in I$,
    let $g_{ij}:=f_{(i\wedge i_0)(j\wedge i_0)}$ if $j\wedge i_0$ exists,
    and let $g_{ij}:=1$ otherwise. 
    One can check that $((Y_i),(g_{ij}))$ is a well-defined inverse system over $I$ and we denote this inverse system
    by $\mathcal{X}_{i_0}$.
    % To be convenient, we always assume $X_{i\wedge i_0}=1$ if $i\wedge i_0$ does not exists, $f_{(i\wedge i_0)(j\wedge i_0)}=1$ if $i\wedge i_0$ or $j\wedge i_0$ does not exists. 
    % Following this convention, $\mathcal{X}_{i_0}=((X_{i\wedge i_0}),(f_{(i\wedge i_0)(j\wedge i_0)}))$.
\end{definition}

\begin{theorem}\label{thm:projection}
    Let $I$ be an in-forest poset,
    let $\mathcal{X}:=((X_i),(f_{ij}))$ be an inverse system over $I$ and 
    let $(X,(p_i)):=\varprojlim(\mathcal{X})$.
    Let $i_0\in I$.
    For $i\in I$,
    let
    $q_i:=f_{(i\wedge i_0)i_0}$ if $i\wedge i_0$ exists, 
    let $q_i=1$ otherwise,
    let $\phi_i:=f_{(i\wedge i_0)i}$ if $i\wedge i_0$ exists, and let $\phi_i=1$ otherwise.
    Let $\Phi:=(\phi_i):\mathcal{X}\rightarrow\mathcal{X}_{i_0}$.
    Then the following holds:
    \begin{itemize}
        \item [\rm (i)] $(X_{i_0},(q_i))=\varprojlim(\mathcal{X}_{i_0})$;
        \item [\rm (ii)] $\Phi$ is a morphism and $\varprojlim(\Phi)=p_{i_0}$;
        \item [\rm (iii)] if $\mathcal{X}$ is surjective, then $p_{i_0}$ is surjective.
    \end{itemize}
\end{theorem}

\begin{proof}
    Let $((Y_i),(g_{ij})):=\mathcal{X}_{i_0}$.

    {\rm (i)} Let $i, j\in I$ such that $i\le j$.
    Note that $q_i:X_{i_0}\rightarrow Y_i$ and $q_j:X_{i_0}\rightarrow Y_j$.
    By Lemma \ref{lem:existenceofmeet}, if $j\wedge i_0$ does not exist, then $i\wedge i_0$
    also does not exist. 
    Hence $g_{ij}\circ q_j=1=q_i$.
    If $j\wedge i_0$ exists, then $i\wedge i_0$ also exists
    and \[g_{ij}\circ q_j=f_{(i\wedge i_0)(j\wedge i_0)}\circ f_{(j\wedge i_0)i_0}=f_{(i\wedge i_0)i_0}=q_i\]
    Hence $g_{ij}\circ q_j=q_i.$
    Let $(Y,(q_i'))$ such that for every $i\in I$, $q_i':Y\rightarrow Y_i$
    and 
    $g_{ij}\circ q_j'=q_i'$ for every $i\le j\in I$.
    If $i\wedge i_0$ exists, $q_i=f_{(i\wedge i_0)i_0}=g_{(i\wedge i_0)i_0}$
    which implies that
    \[q_i\circ q_{i_0}'=g_{(i\wedge i_0)i_0}\circ q'_{i_0}=q_{i\wedge i_0}'=g_{(i\wedge i_0)i}\circ q_i'=f_{(i\wedge i_0)(i\wedge i_0)}\circ q_i'=q_i'.\]
    If $i\wedge i_0$ does not exist, then $q_i\circ q_{i_0}'=1=q_{i}'$.
    Hence $q_i\circ q_{i_0}'=q_i'$ for every $i\in I$ (see Figure~\ref{fig:projection}).
    \begin{figure}[ht]
        \centering
        \[\xymatrix{
        &&Y_i\\
        Y\ar[r]|{q'_{i_0}}\ar@/^/[rru]|{q'_i}\ar@/_/[rrd]|{q'_j}& X_{i_0}\ar[ru]_{q_i}\ar[rd]^{q_j}\\
        &&Y_j
    }\]
        \caption{Diagram for the proof of Theorem~\ref{thm:projection}}
        \label{fig:projection}
    \end{figure}
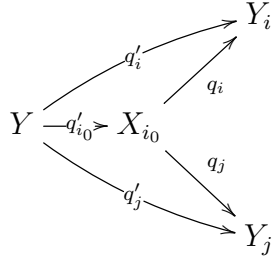
    Let $h$ be a homomorphism from $Y$ to $X_{i_0}$
    satisfying $q_i\circ h=q_i'$ for every $i\in I$. 
    Note that $q_{i_0}={\rm id}_{X_{i_0}}$. 
    Hence $h=q_{i_0}\circ h=q_{i_0}'$.
    This implies that $q_{i_0}'$ is unique with respect to the property that
    $q_i\circ q_{i_0}'=q_i'$ for every $i\in I$.
    We thereby deduce that $(X_{i_0},(q_i))=\varprojlim(\mathcal{X}_{i_0})$.

    {\rm (ii)}
    Let $i,j\in I$ such that $i\le j$.
    By Lemma \ref{lem:existenceofmeet}, $j\wedge i_0$ exists if and only if $i\wedge i_0$ exists.
    If both $j\wedge i_0$ and $i\wedge i_0$ do not exist, then $g_{ij}=1$ and $\phi_i=1$.
    Hence $g_{ij}\circ\phi_j=1=\phi_i\circ f_{ij}$.
    If both $j\wedge i_0$ and $i\wedge i_0$ exist, then
    \[g_{ij}\circ \phi_j=f_{(i\wedge i_0)(j\wedge i_0)}\circ f_{(j\wedge i_0)j}=f_{(i\wedge i_0)j}=f_{(i\wedge i_0)i}\circ f_{ij}=\phi_i\circ f_{ij}.\]
    Hence $\Phi$ is a morphism. 
    We now show that $\varprojlim(\Phi)=p_{i_0}$.
    If $i\wedge i_0$ exists, then 
    \[q_i\circ p_{i_0}=f_{(i\wedge i_0)i_0}\circ p_{i_0}=p_{i\wedge i_0}=\phi_i\circ p_i.\]
    If $i\wedge i_0$ does not exist, then
    \[q_i\circ p_{i_0}=1=\phi_i\circ p_i.\]
    Hence the equation $q_i\circ p_{i_0}=\phi_i\circ p_i$ always holds. 
    This implies that $\varprojlim(\Phi)=p_{i_0}$.

    {\rm (iii)} Since $\mathcal{X}$ is a surjective inverse system, 
    $f_{ij}$ is surjective for all $i\le j\in I$.
    Note that $\phi_i:X_i\rightarrow Y_i$ for $i\in I$.
    If $i\wedge i_0$ exists, then $\phi_i=f_{(i\wedge i_0)i}$ is surjective.
    Otherwise, $Y_i=1$ and it is clear that $\phi_i$ is surjective.
    Hence,
    $\phi_i$ is surjective for every $i\in I$.
    Let $i,j\in I$ such that $i\le j$.
    By Lemma \ref{lem:existenceofmeet},
    $i\wedge i_0$ exists if and only if $j\wedge i_0$ exists.
    If both $i\wedge i_0$ and $j\wedge i_0$ do not exist,
    then $\ker(\phi_i)=X_i$ and $\ker(\phi_j)=X_j$.
    In this case, $f_{ij}|_{\ker(\phi_j),\ker(\phi_i)}=f_{ij}$ which is surjective.
    If both $i\wedge i_0$ and $j\wedge i_0$ exist, then by Lemma \ref{lem:existenceofmeet},
    $i=i\wedge i_0$ or $i\wedge i_0=j\wedge i_0$.
    If $i=i\wedge i_0$, then 
    \[\ker(\phi_i)=\ker(f_{(i\wedge i_0)i})=\ker(f_{ii})=1.\]
    Hence $f_{ij}|_{\ker(\phi_j),\ker(\phi_i)}$ is surjective.
    If $i\wedge i_0=j\wedge i_0$, then 
    \[\phi_j=f_{(j\wedge i_0)j}=f_{(i\wedge i_0)j}=f_{(i\wedge i_0)i}\circ f_{ij}=\phi_i\circ f_{ij}\]
    Hence $\ker(\phi_j)=\ker(\phi_i\circ f_{ij})=f_{ij}^{-1}(\ker(\phi_i))$.
    % \[\ker(\phi_i)=\ker(f_{(i\wedge i_0)i})=\ker(f_{(j\wedge i_0)i}).\]
    Since $f_{ij}$ is surjective,
    we have that 
    $f_{ij}(\ker(\phi_j))=\ker(\phi_i).$
    Therefore, $f_{ij}|_{\ker(\phi_j),\ker(\phi_i)}$ is surjective 
    which implies that $\ker(\Phi)$ is a surjective inverse system.
    Then by Theorem \ref{thm:limsurj} and {\rm (ii)}, $p_{i_0}$ is surjective.
\end{proof}

Theorem~\ref{thm:projection} establishes that a limit projection $p_{i_0}$ of a surjective inverse system over an in-forest poset is also surjective.
This does not hold in general for an arbitrary inverse system. Moreover, the theorem shows that
$p_{i_0}$ can be realised as the inverse limit of a certain morphism, allowing us to compute $\ker(p_{i_0})$ using Theorem~\ref{thm:limpullback}, which is essential in our discussion.
\section{Hybrid wreath products}\label{sec:hybridwreath}

In this section, we introduce a new concept,
which we call hybrid wreath product. 
Throughout this section, all sets are finite.

\begin{definition}\label{def:gw}
    Let $G$ and $H$ be groups and let
    $\theta:G\rightarrow H$ be a homomorphism.
    Let $\Omega:=[H:\theta(G)]$ be the set of right cosets of $\theta(G)$ in $H$ and
    let $\rho:H\rightarrow\Sym(\Omega)$ be the 
    action of $H$ on $\Omega$ by right multiplication.
    Write $\omega=\theta(G)\in\Omega$ and let $(t_\nu)_{\nu\in\Omega}$ 
    be a permutation transversal with respect to $\omega$. Then $H_\omega=\theta(G)$.
    Let $\iota:H\rightarrow \theta(G)\wr_\Omega \rho(H)$ 
    be the standard embedding map of $H$
    with respect to $\omega$ and $(t_\nu)_{\nu\in\Omega}$
    and let
    \[\tilde{\theta}:=\theta|_{G,\theta(G)}\wr_{{\rm id}_\Omega}{\rm id}_{\rho(H)}:G\wr_{\Omega}\rho(H)\rightarrow\theta(G)\wr_{\Omega}\rho(H).\]
    Define \[\operatorname{HW}(G,H,\theta,\iota):=\tilde{\theta}^{-1}(\iota(H))\le G\wr_\Omega\rho(H)\]
    which is called the \emph{hybrid wreath product} of $G$ by $H$
    with respect to $\theta$ and $\iota$. 
    We call the map \[p_\theta:=(\iota|_{H,\iota(H)})^{-1}\circ\tilde{\theta}:\operatorname{HW}(G,H,\theta,\iota)\rightarrow H\]
    the \emph{standard map} of the hybrid wreath product.
    If $\theta(G)\trianglelefteq H$,
    then we say that the hybrid wreath product is \emph{normal}.
    Also define 
    \[\operatorname{BW}(G,H,\theta,\iota):=p_\theta^{-1}(\theta(G))\le \operatorname{HW}(G,H,\theta,\iota)\]
    which is called the \emph{base subgroup} of $\operatorname{HW}(G,H,\theta,\iota)$.
    For every $\nu\in\Omega$, define an \emph{evaluation map} of $\operatorname{BW}(G,H,\theta,\iota)$ as follows:
    $p_\nu:\operatorname{BW}(G,H,\theta,\iota)\rightarrow G, p_\nu(f,\sigma)=f(\nu)$.
\end{definition}

Now we give an example to illustrate the definition of hybrid wreath product, but
first we introduce some notation for extensions of groups.

\begin{notation}
    Given a group $X$, we write $X=A.B$ to denote that $X$ is an \emph{extension of $B$ by $A$}, that is, \( X \) has a normal subgroup isomorphic to \( A \), and the corresponding quotient is isomorphic to \( B \).
    In addition, if $A\le N\le A.B$  
    such that $N/A\cong C\le B$, then we write $N=A.C$.
\end{notation}

\begin{example}\label{examp:hybrid}
    Let $G:=F_{21}=\mathbb{Z}_7\rtimes\mathbb{Z}_3$ be the Frobenius group of order $21$,
    and let $H:=\operatorname{S}_3=\mathbb{Z}_3\rtimes\mathbb{Z}_2$ be the symmetric group of degree $3$.
    Write $\operatorname{S}_3=\langle a\rangle\rtimes\langle b\rangle$. 
    Let $\theta:G\rightarrow H$ be a homomorphism such that $\theta(G)=\langle a\rangle$, and
    let $\Omega=\{\langle a\rangle,\langle a\rangle b\}$
    be the set of right cosets of $\theta(G)$ in $H$.  
    Note that the set $\{1,b\}$ forms a permutation transversal with respect to $\langle a\rangle$ so
    we identify $\Omega$ with this set of representatives.

    Let $\iota:H\rightarrow\theta(G)\wr_\Omega\rho(H)$ be the standard embedding of $H$ with respect to $\omega$ and $\{1,b\}$.
    Observe that $\iota$ is the homomorphism sending $a$ to $(f_a,{\rm id}_\Omega)$
    and sending $b$ to $(f_b,(1,b))$,
    where $f_a(1)=a$, $f_a(b)=a^{-1}$, $f_b(1)=1$ and $f_b(b)=1$.
    This allows us to compute $\operatorname{HW}(G,H,\theta,\iota)$
    and $\operatorname{BW}(G,H,\theta,\iota)$.

    It is clear that $\operatorname{HW}(G,H,\theta,\iota)$ has a normal series
    \[1\trianglelefteq \ker(p_\theta)\trianglelefteq\operatorname{BW}(G,H,\theta,\iota)\trianglelefteq\operatorname{HW}(G,H,\theta,\iota).\]
    One can verify that $\ker(p_\theta)\cong\mathbb{Z}_7^2$
    and $\operatorname{BW}(G,H,\theta,\iota)=\{(x,y)\in G^2\mid \theta(x)=\theta(y)^{-1}\}$,
    which is a proper subdirect subgroup of $G^2$.
    Thus, we have \[\operatorname{HW}(G,H,\theta,\iota)\cong \mathbb{Z}_7^2.\operatorname{S}_3\]
    and \[\operatorname{BW}(G,H,\theta,\iota)\cong \mathbb{Z}_7^2.\mathbb{Z}_3\le G^2.\]
\end{example}

\begin{lemma}\label{lem:uniquewreath}
    Let $G$ and $H$ be groups and let
    $\theta:G\rightarrow H$ be a homomorphism.
    Let $\Omega:=[H:\theta(G)]$ be the set of right cosets of $\theta(G)$ in $H$ and
    let $\rho:H\rightarrow\Sym(\Omega)$ be the 
    action of $H$ on $\Omega$ by right multiplication.
    Write $\omega=\theta(G)\in\Omega$, and let $\iota$ and $\lambda$
    be standard embeddings of $H$ with respect to $\omega$ (and potentially different
permutation transversals).
    Then there exists $x\in G^\Omega\le G\wr_\Omega \rho(H)$ such that $\operatorname{HW}(G,H,\theta,\iota)=\operatorname{HW}(G,H,\theta,\lambda)^x$.
\end{lemma}

\begin{proof}
    Let $S:=\rho(H)$.
    By Lemma \ref{lem:uniqueembedding}, there exists $x_1\in \theta(G)^\Omega\le \theta(G)\wr_\Omega S$ such that
    $\lambda=\Inn_{\theta(G)\wr_\Omega S}(x_1)\circ\iota$. 
    Therefore, $\iota(H)=\lambda(H)^{x_1}$.
    By Lemma \ref{lem:wrhomo}, $\tilde{\theta}$ is a surjective homomorphism.
    With the identification of $\theta(G)^\Omega$ and $\theta(G)^\Omega\rtimes 1$,
    let $x\in G\wr_\Omega S$ such that $\tilde{\theta}(x)=x_1$ and
    write $x=(f,\sigma)$.    
    Then by the definition of $\tilde{\theta}$, 
    $\tilde{\theta}(x)=\tilde{\theta}(f,\sigma)=(\theta\circ f,\sigma)=x_1$.
    Therefore $\sigma=1$ which forces $x\in G^\Omega$.
    We claim that $\operatorname{HW}(G,H,\theta,\iota)=\operatorname{HW}(G,H,\theta,\lambda)^x$. 
    Note that \[
    \begin{aligned}
        \tilde{\theta}(\operatorname{HW}(G,H,\theta,\lambda)^x)&=\tilde{\theta}(\operatorname{HW}(G,H,\theta,\lambda))^{\tilde{\theta}(x)}\\
        &=\lambda(H)^{x_1}=\iota(H)=\tilde{\theta}(\operatorname{HW}(G,H,\theta,\iota)).
    \end{aligned}
    \] By the correspondence theorem, we deduce $\operatorname{HW}(G,H,\theta,\iota)=\operatorname{HW}(G,H,\theta,\lambda)^x$.
    
\end{proof}

Lemma \ref{lem:uniquewreath} shows that the hybrid wreath product is independent of the choice of the permutation transversal, up to conjugation.
Throughout the remainder of this paper, we 
write $\operatorname{HW}(G,H,\theta)$ instead of $\operatorname{HW}(G,H,\theta,\iota)$ (if following the notation in Definition \ref{def:gw}).
Also, we write $\operatorname{BW}(G,H,\theta)$ instead of $\operatorname{BW}(G,H,\theta,\iota)$.

The base subgroup $\operatorname{BW}(G,H,\theta)$ of a normal hybrid wreath product admits a concise characterisation via inverse systems (see Lemma~\ref{lem:bw}).
This representation captures the structure of $\operatorname{BW}(G,H,\theta)$ 
and is essential for the proof of Theorem~\ref{thm:comp_l}.

\begin{lemma}\label{lem:bw}
    Assume the notation in Definition \ref{def:gw}. 
    Let $|\Omega|=n$, let $\Omega=\{\nu_1,\ldots,\nu_n\}$ and let $p_\theta$
    be the standard map of $\operatorname{HW}(G,H,\theta)$.
    Let $\mathcal{X}$ be the inverse system given by
    \[\xymatrix@=0.5em{
        G\ar@{.}[dd]_n\ar[rd]^{\Inn(t_{\nu_1}^{-1})\circ \theta}&\\
        &\theta(G)\\
        G\ar[ru]_{\Inn(t_{\nu_n}^{-1})\circ \theta}
    }\]
    If $\operatorname{HW}(G,H,\theta)$ is a normal hybrid wreath product,
    then 
    \begin{itemize}
        \item [\rm (i)] $\operatorname{BW}(G,H,\theta)=\{(f,1)\in G\wr\rho(H)\mid \exists h\in\theta(G): (\theta\circ f,1)=\iota(h)\}$;
        \item [\rm (ii)] for every $\nu\in\Omega$, the evaluation map $p_\nu$ is a surjective homomorphism;
        \item [\rm (iii)] $\operatorname{BW}(G,H,\theta)$, equipped with $(p_{\nu_i})_{1\le i\le n}$ and $p_\theta$,
        is the inverse limit of the inverse system $\mathcal{X}$.

    \end{itemize}
\end{lemma}
\begin{proof}
    {\rm (i)} 
    Let $B$ be the base subgroup of $\theta(G)\wr\rho(H)$ and let
    $\tilde{B}$ be the base subgroup of $G\wr\rho(H)$.
    It can be seen that $\tilde{\theta}^{-1}(B)=\tilde{B}$. 
    Since $\theta(G)\trianglelefteq H$, we have that $\ker(\rho)=\theta(G)$.
    Then for every $h\in \theta(G)$, it holds that $\iota(h)=(f_h,\rho(h))=(f_h,1)\in B$.
    This implies that $\iota(\theta(G))\le B$.
    Therefore, $\operatorname{BW}(G,H,\theta)=\tilde{\theta}^{-1}(\iota(\theta(G)))\le \tilde{B}$.
    Let $(f,1)\in \operatorname{BW}(G,H,\theta)$.
    Then $\tilde{\theta}(f,1)=(\theta\circ f,1)\in \iota(\theta(G))$ which implies
    that $(\theta\circ f, 1)=\iota(h)$ for some $h\in\theta(G)$.
    Therefore, \[\operatorname{BW}(G,H,\theta)\le \{(f,1)\in G\wr\rho(H)\mid (\theta\circ f,1)=\iota(h),\mbox{ for some }h\in \theta(G) \}.\]
    In the other direction, note that $(\theta\circ f,1)=\iota(h)$ for some $h\in\theta(G)$
    implies that $(f,1)\in \tilde{\theta}^{-1}(\iota(\theta(G)))=\operatorname{BW}(G,H,\theta)$.
    We thereby have 
    \[\operatorname{BW}(G,H,\theta)\ge \{(f,1)\in G\wr\rho(H)\mid (\theta\circ f,1)=\iota(h),\mbox{ for some }h\in \theta(G) \}.\]
    This completes the proof of {\rm (i)}.

    {\rm (ii)} Let $\nu\in\Omega$. By {\rm (i)}, let $(f_1,1),(f_2,1)\in\operatorname{BW}(G,H,\theta)$.
    Then $p_\nu((f_1,1)(f_2,1))=p_\nu((f_1f_2,1))=f_1f_2(\nu)=f_1(\nu)f_2(\nu)=p_\nu((f_1,1))p_\nu((f_2,1))$.
    Hence $p_\nu$ is a homomorphism. 
    Now we show that $p_\nu$ is surjective.
    Let $g\in G$ and let $h:=t_\nu^{-1} \theta(g)t_\nu$.
    Since $\theta(G)\trianglelefteq H$, we have that $h\in \theta(G)$.
    Recall that $\ker(\rho)=\theta(G)$.
    Then by Definition \ref{def:universalembedding}, 
    $\iota(h)=(f_h,1)$ where $f_h(\mu)=t_\mu ht_{\mu^y}^{-1}=t_\mu ht_{\mu}^{-1}$ for every $\mu\in\Omega$.
    In particular, $f_h(\nu)=\theta(g)$.
    For every $\mu\in\Omega$, there exists an element $k_\mu\in G$
    such that $\theta(k_\mu)=f_h(\mu)$.
    Let $f\in G^\Omega$ be given by $f(\mu)=k_\mu$ for $\mu\ne\nu$ 
    and $f(\nu)=g$.
    Then $\theta\circ f=f_h$ which implies that $(\theta\circ f,1)=\iota(h)$.
    Therefore, $(f,1)\in \operatorname{BW}(G,H,\theta)$ by {\rm (i)}.
    This follows that $p_\nu((f,1))=f(\nu)=g$.
    This completes the proof of {\rm (ii)}.

    {\rm (iii)} 
    Let $(f,1)\in \operatorname{BW}(G,H,\theta)$ and let $g:=p_\theta((f,1))$.
    By Definition \ref{def:gw}, 
    $g\in\theta(G)=\ker(\rho)$ and $\iota(g)=\tilde{\theta}((f,1))=(\theta\circ f,1)$.
    Then for every $\nu\in\Omega$, $(\theta\circ f)(\nu)=t_{\nu}gt_{\nu^g}^{-1}=t_{\nu}gt_{\nu}^{-1}$ 
    and
    \[\begin{aligned}
        \Inn(t_{\nu}^{-1})\circ\theta\circ p_{\nu}((f,1))&=\Inn(t_{\nu}^{-1})(t_{\nu}gt_{\nu}^{-1})\\
        &=p_\theta((f,1))
    \end{aligned}
    \]
    Therefore, we have $\Inn(t_{\nu_i}^{-1})\circ\theta\circ p_{\nu_i}=p_\theta$ for $i\in\{1,\ldots,n\}$.
    Let $X$ be a group equipped with homomorphisms 
    $q_i:X\rightarrow G$ for $i\in\{1,\ldots,n\}$ 
    and $q_\theta:X\rightarrow \theta(G)$,
    such that $\Inn(t_{v_i}^{-1})\circ\theta\circ q_i=q_\theta$ for every $i\in\{1,\ldots,n\}$.
    Let $x\in X$. Define $g_x\in G^\Omega$ via
    $g_x(\nu_i):=q_i(x)$ for every $i\in\{1,\ldots,n\}$.
    Then \[\begin{aligned}
        (\theta\circ g_x)(\nu_i)&=(\theta\circ q_i)(x)\\
        &=(\Inn(t_{\nu_i})\circ \Inn(t_{\nu_i}^{-1})\circ \theta\circ q_i)(x)\\
        &=(\Inn(t_{\nu_i})\circ q_\theta)(x)\\
        &=t_{\nu_i}(q_\theta(x))t_{\nu_i}^{-1}\\
        &=t_{\nu_i}(q_\theta(x))t_{\nu_i^{q_\theta(x)}}^{-1}
    \end{aligned}
    \]
    Therefore, by the definition of $\iota$, 
    $\iota(q_\theta(x))=(\theta\circ g_x,1)$.
    Then by {\rm (i)},
    $(g_x,1)\in \operatorname{BW}(G,H,\theta)$.
    Let $u:X\rightarrow\operatorname{BW}(G,H,\theta)$ be given by $u(x):=(g_x,1)$.
    Let $x_1,x_2\in X$. Note that $g_{x_1x_2}(\nu_i)=q_i(x_1x_2)=q_i(x_1)q_i(x_2)=g_{x_1}(\nu_i)g_{x_2}(\nu_i)$
    for every $i\in\{1,\ldots,n\}$. 
    Then $g_{x_1x_2}=g_{x_1}g_{x_2}$
    which implies that 
    \[u(x_1x_2)=(g_{x_1x_2},1)=(g_{x_1},1)(g_{x_2},1)=u(x_1)u(x_2).\]
    Therefore, $u$ is a homomorphism. 
    It can be seen that $p_{\nu_i}\circ u=q_i$ (see Figure~\ref{fig:lembw2}).
    \begin{figure}
        \centering
        \[\xymatrix@=0.5em{
            &&&&&&&G\ar@{.}[dddd]_(.3)n\ar[rdd]^{\Inn(t_{\nu_1}^{-1})\circ \theta}&\\
            &&&&&&&&&&\\
            X\ar@{-->}[rrr]^{u\ \ \ \ \ }\ar@/^/@{-->}[rrrrrruu]^{q_1}\ar@/_/@{-->}[rrrrrrdd]_{q_n}&&&\operatorname{BW}(G,H,\theta)\ar@{-->}[rrrruu]^{p_{\nu_1}}\ar@{-->}[rrrrdd]_{p_{\nu_n}}\ar@{-->}[rrrrr]^{p_\theta}&&&& &\theta(G)\\
            &&&&&&&&&&&&\\
            &&&&&&&G\ar[ruu]_{\Inn(t_{\nu_n}^{-1})\circ \theta}
        }\]
        \caption{Diagram for the proof of Lemma \ref{lem:bw}}
        \label{fig:lembw2}
    \end{figure}
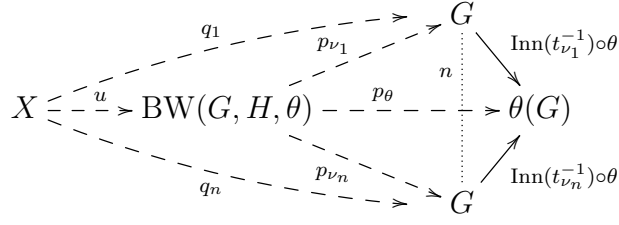
    Let $u_1:X\rightarrow \operatorname{BW}(G,H,\theta)$ be a homomorphism
    such that $p_{\nu_i}\circ u_1=q_i$.
    For $x\in X$, let $u_1(x):=(g'_x,1)$. Then 
    $g'_x(\nu_i)=(p_{\nu_i}\circ u_1)(x)=q_i(x)=g_x(\nu_i)$
    for $i\in\{1,\ldots,n\}$. Hence $g'_x=g_x$ and $u_1=u$.
    By the definition of inverse limits, this concludes the proof.
\end{proof}
\section{Motivation for main results}
\label{sec:appendix}
This section is a sketch of the ideas of the proof of Theorem \ref{thm:comp_l}. 
It is written in a somewhat informal manner and for illustration purpose only.

\begin{notation}
    We use subscripts to distinguish extensions,
    for example, $A._1B$ and $A._2B$ denote two different extensions of $B$ by $A$.
    In addition, if $A\le N\le A._1B$  
    such that $N/A\cong C\le B$, then we write $N=A._1C$.
\end{notation}

\subsection{Normal series of length 2}\label{subsec:l2}

Let $L_1$ and $L_2$ have compatible normal series of length $2$.
Write $L_1=A._1B$ and $L_2=A._2B$.
Let $\pi_1:L_1\rightarrow B$ and $\pi_2:L_2\rightarrow B$ be the natural projections.
Let $\mathcal{X}$ be the inverse system shown in the right half of Figure \ref{fig:length2}
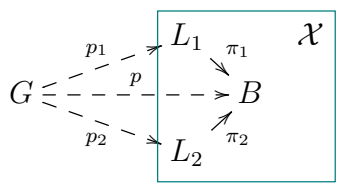
\begin{figure}
    \[\xymatrix@=0.5em{
        &&&&L_1\ar[rd]^{\pi_1}&&\mathcal{X}\\
        G\ar@{-->}[rrrru]^{p_1}\ar@{-->}[rrrrd]_{p_2}\ar@{-->}[rrrrr]^{p}&&&&&B&\\
        &&&&L_2\ar[ru]_{\pi_{2}}&&
        \save "1,5"."3,7"*+[F:teal]\frm{}\restore
}\]
\caption{Inverse system $\mathcal{X}$ and its inverse limit}
\label{fig:length2}
\end{figure}
and let $G$ be the inverse limit of $\mathcal{X}$ associated with 
the maps $p_1:G\rightarrow L_1$, $p_2:G\rightarrow L_2$ and $p:G\rightarrow B$ (see Figure \ref{fig:length2}).
By Theorem \ref{thm:projection} {\rm (iii)}, $p_1$ and $p_2$ are surjective
and with the aid of Theorem \ref{thm:projection} {\rm (i)} and Theorem \ref{thm:limpullback}, 
one can check that $\ker(p_1)\cong A\cong\ker(p_2)$. We also point out that these properties can alternatively be verified by considering the explicit construction
\[G=\{(l_1,l_2,b)\in L_1\times L_2\times B\mid \pi_1(l_1)=\pi_2(l_2)=b\}.\]

Therefore, $(G,p_1,p_2)$ is a witness system for the compatibility of $(L_1,L_2)$.
See Lemma \ref{lem:length2} for a more formal proof along these lines.

\subsection{Normal series of length 3: method based on mixed groups}\label{subsec:l3}
Now we move on to the case of normal series of length $3$.
We first discuss a method based on `mixed groups'.
This method is not used in the proof of our main result and is less powerful,
but it inspired our research and is easier to understand.

Let $L_1$ and $L_2$ have compatible normal series of length $3$.
Write $L_1=A._1B._1C$ and $L_2=A._2B._2C$. Let $\pi_{11}:L_1\rightarrow B._1C$, $\pi_{22}:L_2\rightarrow B._2C$, $\tau_1:B._1C\rightarrow C$
and $\tau_2:B._2C\rightarrow C$ be the natural projections.
Let $M_1$ and $M_2$ also have normal series with the same factors in the same order as $L_1$ and $L_2$.
Suppose further that $M_1$ and $M_2$ satisfy additional `mixed' conditions: there exist surjective homomorphisms $\pi_{21}:M_1\rightarrow B._1C$ and $\pi_{12}:M_2\rightarrow B._2C$
such that $\pi_{21}^{-1}(B)\cong \pi_{22}^{-1}(B)$ and $\pi_{12}^{-1}(B)\cong\pi_{11}^{-1}(B)$.
Then we say that $M_1$ and $M_2$ are \emph{mixed groups} of $L_1$ and $L_2$ and we write $M_1=A._2B._1C$ and $M_2=A._1B._2C$.

Let $\mathcal{X}$ be the inverse system shown in the right half of Figure \ref{fig:length3}
\begin{figure}
    \[\xymatrix@=0.5em{
    &&&&L_1\ar[rd]^{\pi_{11}}&&\mathcal{X}&&&&\\
    &&&&&B._1C\ar[rdd]^{\tau_1}&&&&&&\\
    &&&&M_1\ar[ru]_{\pi_{21}}&&&&&\\
    G\ar@{-->}[rrrrddd]_{p_2}\ar@{-->}[rrrruuu]^{p_1}&&&&&&C\\
    &&&&M_2\ar[rd]^{\pi_{12}}&&&&&\\
    &&&&&B._2C\ar[ruu]^{\tau_2}&&&&&\\
    &&&&L_2\ar[ru]_{\pi_{22}}&&&&&
    \save "1,5"."7,7"*++[F:teal]\frm{}\restore
}\]
\caption{Inverse system $\mathcal{X}$ and its inverse limit}
\label{fig:length3}
\end{figure}
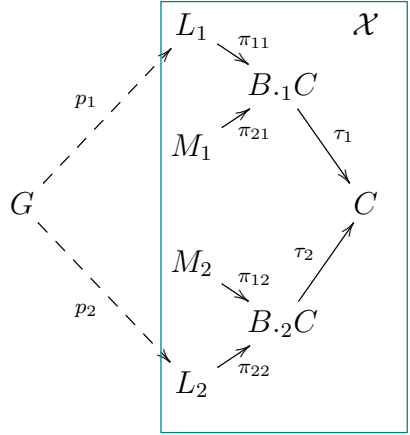
and let $G$ be the inverse limit of $\mathcal{X}$ with associated maps $p_1:G\rightarrow L_1$, $p_2:G\rightarrow L_2$.
(Here and in the remainder of this section, we always just name the limit projections we need.)
For $\delta\in\{1,2\}$, let $\mathcal{Y}_\delta$ be the inverse system given by
\[\xymatrix@=0.5em{
    A._1B\ar[rrrrrd]^{\pi_{1\delta}|_{A._1B,B}}\\
    &&&&&B\\
    A._2B\ar[rrrrru]_{\pi_{2\delta}|_{A._2B,B}}
}\]
and let $K_\delta$ 
be the inverse limit of $\mathcal{Y}_\delta$.

By Theorem \ref{thm:projection} {\rm (iii)}, $p_1$ and $p_2$ are surjective
and using Theorem \ref{thm:projection} {\rm (i)} and Theorem \ref{thm:limpullback}, 
one can check that 
$\ker(p_1)\cong A\times K_1$ and $\ker(p_2)\cong A\times K_2$.
If $K_1\cong K_2$, then $(G,p_1,p_2)$ is a witness system for the compatibility of $(L_1,L_2)$.
Therefore, to turn this argument into a proof for the compatibility of $L_1$ and $L_2$, we would need positive answers to the following two questions:

\begin{itemize}
    \item [\rm (i)] Do `mixed groups' like $M_1$ and $M_2$ exist?
    \item [\rm (ii)] If $M_1$ and $M_2$ exist, does $K_1\cong K_2$ hold?
\end{itemize}

The answers are sometimes positive and sometimes negative. 
The following is an example where one of $M_1$ and $M_2$ does not exist.

\begin{example}\label{examp:F21S3}
    Let $L_1:=\operatorname{F}_{21}\times \mathbb{Z}_2$ and let $L_2:=\mathbb{Z}_7\times \operatorname{S}_3$.
    One can check that both $L_1$ and $L_2$ are of the form $\mathbb{Z}_7.\mathbb{Z}_3.\mathbb{Z}_2$
    so they have compatible normal series of length $3$.
    Write $L_1=(\mathbb{Z}_7)._1(\mathbb{Z}_3)._1(\mathbb{Z}_2)$ and $L_2=(\mathbb{Z}_7)._2(\mathbb{Z}_3)._2(\mathbb{Z}_2)$.
    Suppose that $M$ is of form $(\mathbb{Z}_7)._1(\mathbb{Z}_3)._2(\mathbb{Z}_2)$.
    By the Schur-Zassenhaus Theorem, $M=\mathbb{Z}_7\rtimes_\varphi \operatorname{S}_3$ for
    some homomorphism $\varphi:\operatorname{S}_3\rightarrow\Aut(\mathbb{Z}_7)$.
    Since $\Aut(\mathbb{Z}_7)$ is cyclic, 
    the subgroup $\mathbb{Z}_3$ of $\operatorname{S}_3$ lies in $\ker(\varphi)$,
    which implies that the subgroup $\mathbb{Z}_3$ centralises $\mathbb{Z}_7$, but this contradicts
    the fact that $\operatorname{F}_{21}=(\mathbb{Z}_7)._1\mathbb{Z}_3$.
    Therefore, the `mixed group' $M$ does not exist.
    On the other hand, there exists a group of the form $\mathbb{Z}_7^2\rtimes\operatorname{S}_3$ (see Example~\ref{examp:hybrid}),
    and as we shall see in Section~\ref{sec:length3hybrid}, this fact 
    allows us to construct a witness system for the compatibility of $L_1$ and $L_2$.
\end{example}

We note that there are also examples violating {\rm (ii)},
but it is difficult to explain this without the formal definition of `mixed groups',
so we will skip this.
We have seen that the method illustrated in this section is not applicable to $L_1$ and $L_2$
from Example \ref{examp:F21S3}.
We will give a stronger method in Section~\ref{sec:length3hybrid} and it is not hard
to prove that this method works for $L_1$ and $L_2$ from Example \ref{examp:F21S3}.
Indeed, this follows immediately from Corollary~\ref{cor:square-free}.

\subsection{Normal series of length 3: method based on hybrid wreath product}\label{subsec:l3+}
\label{sec:length3hybrid}

Let $B$ be a group and let $\alpha_1,\ldots,\alpha_n\in\Aut(B)$.
Then $\{(\alpha_1(b),\ldots,\alpha_n(b))\in B^n|b\in B\}$ is a \emph{full diagonal subgroup} of $B^n$.
Now given a group $A._1B$, we define a new group $A^n._1B$ as follows.
\begin{notation}\label{not:hybrid}
    Let $X:=A._1B$ and 
    let $\eta:X^n\rightarrow B^n$ be the natural projection.
    Let $\alpha_1,\ldots,\alpha_n\in\Aut(B)$ and let
    $D:=\{(\alpha_1(b),\ldots,\alpha_n(b))\in B^n|b\in B\}\le B^n$ be a full diagonal subgroup.
    Let $Y:=\eta^{-1}(D)$
    be the full preimage of $D$.
    We then write $Y=A^n._{1}B$.
\end{notation}

Adopting the notation above, we have
that for every $i\in\{1,\ldots,n\}$, 
\[Y/(A\times\cdots \times A\times \underset{i\mbox{-th}}{1}\times A\times\cdots \times A)\cong X.\]
Let $\pi:X\rightarrow B$ be the natural projection, and for $i\in\{1,\ldots,n\}$, let $\eta_i:Y\rightarrow X$ be the natural projection with respect to $A\times\cdots \times A\times \underset{i\mbox{-th}}{1}\times A\times\cdots \times A$.
Note that $\eta=\underbrace{\pi\times\cdots\times\pi}_n$.
One can also check that  
$Y$ equipped with $\eta_1,\ldots,\eta_n,\eta$ is the inverse limit of the inverse system shown in the right part of Figure \ref{fig:invconvention}.
\begin{figure}
    \[\xymatrix@=0.5em{
    &&&&&X\ar@{.}[dd]_(.3)n\ar[rd]^{\alpha_1\circ\pi}&\\
    Y\ar@{-->}[rrrrru]^{\eta_1}\ar@{-->}[rrrrrd]_{\eta_n}\ar@{-->}[rrrrrr]^{\eta}&&&&&&B\\
    &&&&&X\ar[ru]_{\alpha_n\circ\pi}&
    \save "1,6"."3,7"*++[F:teal]\frm{}
    \restore
}\]

\caption{$Y$ as the inverse limit of an inverse system}
\label{fig:invconvention}
\end{figure}
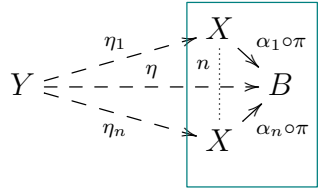

\begin{example}\label{examp:hw}
    Let $G=A._1B$ and let $H=B._2C$.
    Let $\pi:G\rightarrow B$ be the natural projection and let 
    $\iota:B\rightarrow H$ be the natural embedding.
    Recall Notation \ref{not:hybrid}.
    Using the definition of hybrid wreath product and Lemma \ref{lem:bw} {\rm (iii)}, 
    one can check that $\operatorname{HW}(G,H,\iota\circ\pi)\cong A^{|C|}._1B._2C$
    and $\operatorname{BW}(G,H,\iota\circ\pi)\cong A^{|C|}._1B$.
    See Example~\ref{examp:hybrid} for an illustrative example for this construction.
\end{example}

Let $L_1:=A._1B._1C$ and let $L_2:=A._2B._2C$.
Instead of assuming that there exist groups $A._2B._1C$ and $A._1B._2C$ as in Section \ref{subsec:l3},
we assume that there exist groups $M_1:=A^n._2B._1C$ and $M_2:=A^n._1B._2C$ for some $n$.
Choosing $n=|C|$, Example \ref{examp:hw} shows that such $M_1$ and $M_2$ exist. (We just let $M_1$ and $M_2$ be the corresponding hybrid wreath products.)
Let $\pi_{11}:L_1\rightarrow B._1C$, $\pi_{21}:M_1\rightarrow B._1C$, $\pi_{12}:M_2\rightarrow B._2C$, 
$\pi_{22}:L_2\rightarrow B._2C$,
$\tau_1:B._1C\rightarrow C$ and $\tau_2:B._2C\rightarrow C$ be the natural projections.
Let $\alpha_{1;1},\ldots,\alpha_{n;1}\in\Aut(B._1C),\alpha_{1;2},\ldots,\alpha_{n;2}\in\Aut(B._2C)$ and let $\mathcal{X}$ be the inverse system shown in Figure~\ref{fig:inv-X}.
(There are $n$-copies of $L_1$ and $n$-copies of $L_2$ in $\mathcal{X}$.)
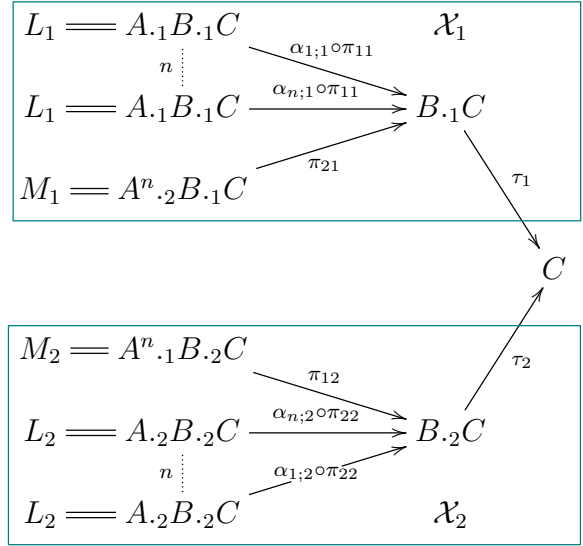
\begin{figure}[ht]
    \centering
    \[\xymatrix@=1.2em{
    L_1\ar@{=}[r]&A._1B._1C\ar@{.}[d]_n\ar[rrrd]^{\alpha_{1;1}\circ\pi_{11}}&&&\mathcal{X}_1\\
    L_1\ar@{=}[r]&A._1B._1C\ar[rrr]^{\alpha_{n;1}\circ\pi_{11}}&&&B._1C\ar[rdd]^{\tau_1}&\\
    M_1\ar@{=}[r]&A^n._2B._1C\ar[rrru]_{\pi_{21}}&&&&&&&\\
    &&&&&C\\
    M_2\ar@{=}[r]&A^n._1B._2C\ar[rrrd]^{\pi_{12}}&&&&&&&\\
    L_2\ar@{=}[r]&A._2B._2C\ar@{.}[d]_n\ar[rrr]^{\alpha_{n;2}\circ\pi_{22}}&&&B._2C\ar[ruu]_{\tau_2}&&\\
    L_2\ar@{=}[r]&A._2B._2C\ar[rrru]|{\alpha_{1;2}\circ\pi_{22}}&&&\mathcal{X}_2&
    \save "1,1"."3,6"*+[F:teal]\frm{}
    \restore
    \save "5,1"."7,6"*+[F:teal]\frm{}
    \restore
}\]
\caption{Inverse system $\mathcal{X}$}
\label{fig:inv-X}
\end{figure}
How to choose the automorphisms $\alpha_{i;\delta}$ is explained in Lemma \ref{lem:comp_l}.

Let $G$ equipped with $p_1:G\rightarrow L_1$, $p_2:G\rightarrow L_2$ be the inverse limit of $\mathcal{X}$.
Let $\mathcal{Y}_1$ be the inverse system given by
\[\xymatrix@=1em{
    A^n._1B\ar[rrrrrd]^{\pi_{12}}\\
    A._2B\ar@{.}[d]_n\ar[rrrrr]_{\alpha_{n;2}\circ\pi_{22}\;}&&&&&B\\
    A._2B\ar[rrrrru]_{\;\alpha_{1;2}\circ\pi_{22}}
}\]
Here we are abusing the notation to denote the restriction of the corresponding transition maps in $\mathcal{X}$.
Similarly, let $\mathcal{Y}_2$ be the inverse system given by
\[\xymatrix@=1em{
    A._1B\ar@{.}[d]_n\ar[rrrrrd]^{\alpha_{1;1}\circ\pi_{11}}\\
    A._1B\ar[rrrrr]^{\alpha_{n;1}\circ\pi_{11}}&&&&&B\\
    A^n._2B\ar[rrrrru]_{\pi_{21}}
}\]
For $\delta\in\{1,2\}$, let $K_\delta$
be the inverse limit of $\mathcal{Y}_\delta$.
One can check that
$\ker(p_1)\cong A^{2n-1}\times K_1$ and $\ker(p_2)\cong A^{2n-1}\times K_2$.
If $K_1\cong K_2$, then $(G,p_1,p_2)$ is a witness system for the compatibility of $(L_1,L_2)$.
In general, whether $K_1\cong K_2$ holds depends on the choice of the automorphisms $\alpha_{i;\delta}$.
We can add extra hypotheses on 
the compatible normal series of $L_1$ and $L_2$ so that suitable $\alpha_{i;\delta}$ exists (see Corollary \ref{cor:compnormalseries}).

\subsection{Longer normal series: a recursive method}
For longer series, we can use a recursive construction.
We illustrate this construction for series of length $3$.

As in Section~\ref{subsec:l3+}, let $L_1:=A._1B._1C$, let $L_2:=A._2B._2C$ and
let $M_1:=A^n._2B._1C$ and $M_2:=A^n._1B._2C$ for some $n$.
Let $\mathcal{X}_1$ be the inverse system given by
\[\xymatrix@=0.5em{
    L_1\ar@{=}[r]&A._1B._1C\ar@{.}[d]_n\ar[rd]\\
    L_1\ar@{=}[r]&A._1B._1C\ar[r]&B._1C\\
    M_1\ar@{=}[r]&A^n._2B._1C\ar[ru]
}\]
and let $\mathcal{X}_2$ be the inverse system given by
\[\xymatrix@=0.5em{
    M_2\ar@{=}[r]&A^n._1B._2C\ar[rd]\\
    L_2\ar@{=}[r]&A._2B._2C\ar@{.}[d]_n\ar[r]&B._2C\\
    L_2\ar@{=}[r]&A._2B._2C\ar[ru]
}\]
For each $i\in\{1,2\}$, there are $n$-copies of $L_i$ in $\mathcal{X}_i$ but the projections to $B._iC$ may be different.
The transition maps of $\mathcal{X}_1$ and $\mathcal{X}_2$ correspond to the transition maps shown in Figure \ref{fig:inv-X}.
For $\delta\in\{1,2\}$, let $\widetilde{L}_\delta$
be the inverse limit of $\mathcal{X}_\delta$.
Let $\mathcal{Y}_1$ be the inverse system given by
\[\xymatrix@=0.5em{
    A^n._1B\ar[rd]\\
    A._2B\ar@{.}[d]_n\ar[r]&B\\
    A._2B\ar[ru]
}\]
where the transition maps are restrictions of the corresponding transition maps in $\mathcal{X}_1$.
Similarly, let $\mathcal{Y}_2$ be the inverse system given by
\[\xymatrix@=0.5em{
    A._1B\ar@{.}[d]_n\ar[rd]\\
    A._1B\ar[r]&B\\
    A^n._2B\ar[ru]
}\]
where the transition maps are restrictions of the corresponding transition maps in $\mathcal{X}_2$.
For $\delta\in\{1,2\}$, let $K_\delta$
be the inverse limit of $\mathcal{Y}_\delta$.
One can check that $\widetilde{L}_\delta=(K_\delta)._\delta (C)$.

If $K_1\cong K_2$, then we reduce to the case of length $2$, as discussed in Section~\ref{subsec:l3}.
Using the method in Section~\ref{subsec:l2}, let $(G,p_1,p_2)$ be the witness system for the compatibility of $(\widetilde{L}_1,\widetilde{L}_2)$.
One can check that there exists $N_\delta\trianglelefteq \widetilde{L}_\delta$ such that: 
$N_\delta\cong A^{2n-1}$, $L_\delta\cong\widetilde{L}_\delta/N_\delta$ and $p_\delta^{-1}(N_\delta)\cong N_\delta\times\ker(p_\delta)$ (see Lemma \ref{lem:length2}).
Let $\pi_\delta:\widetilde{L}_\delta\rightarrow L_\delta$ be the natural projection with respect to $N_\delta$.
Then $(G,\pi_1\circ p_1,\pi_2\circ p_2)$ is a witness system for the compatibility of $(L_1,L_2)$.

For longer series, we follow the same idea.
Let $\mathcal{C}_2$ be the collection of all compatible normal series of length $2$.
Let $\ell\ge 3$ and let $\mathcal{C}_\ell$ be a collection of compatible normal
series of length $\ell$ satisfying a certain extra hypothesis (see Corollary \ref{cor:compnormalseries}).
Given $((A_1)._1\ldots._1(A_\ell),(A_1)._2\ldots._2(A_\ell))\in\mathcal{C}_\ell$, using the method illustrated above,
we define $K_1$ and $K_2$. 
The extra hypothesis mentioned above guarantees that $K_1\cong K_2$.
Now we prove that $((K_1)._1(A_3)._1\ldots.(A_\ell),(K_2)._2(A_3).\ldots._2(A_\ell))\in\mathcal{C}_{\ell-1}$, and
then we prove that if for all pairs in $\mathcal{C}_{\ell-1}$, we can construct
a `good' witness system, then we can construct a `good' witness system for $((A_1)._1\ldots._1(A_\ell),(A_1)._2\ldots._2(A_\ell))$
(for the definition of `good', see Definition \ref{def:goodwit}). 
This guarantees that the induction works for all series in each $\mathcal{C}_\ell$.

\section{Main results}\label{sec:mainresult}
In this section, we present and prove our main result, Corollary~\ref{cor:compnormalseries}, which provides a sufficient condition for two groups to be compatible.
%To illustrate the proof, we refer to Section~\ref{sec:appendix} for the motivation.

\subsection{Group sequences}
We first introduce a concept equivalent to normal series, 
which is more convenient in the context of inverse systems.
\begin{definition}\label{def:groupsequence+surjectivesequence}
    Let $\ell$ be an integer. 
    Let $S_{0}:=1$, for every $i\in\{1,\ldots,\ell\}$, 
    let $S_{i}$ be a group and let $\pi_i:S_{i}\rightarrow S_{i-1}$ be a homomorphism.
    We say that 
    \[((S_{i})_{0\le i\le\ell},(\pi_{i})_{1\le i\le\ell})\]
    is a \emph{group sequence} of length $\ell$.
    If for every $i\in\{1,\ldots,\ell\}$, $\pi_i:S_{i}\rightarrow S_{i-1}$ is a surjective homomorphism, then we say that
    the group sequence is \emph{surjective}.
    
    Let $S_1:=((S_{i;1})_{0\le i\le\ell},(\pi_{i;1})_{1\le i\le\ell})$ and $S_2:=((S_{i;2})_{0\le i\le \ell},(\pi_{i;2})_{1\le i\le \ell})$
    be surjective group sequences of length $\ell$.
    If for every $i\in\{1,\ldots,\ell\}$, $\ker(\pi_{i;1})\cong\ker(\pi_{i;2})$,
    then we say that $S_1$ and $S_2$ are \emph{compatible group sequences}.
\end{definition}

For convenience, we sometimes write $((S_{i}),(\pi_{i}))$ 
rather than $((S_{i})_{0\le i\le\ell},(\pi_{i})_{1\le i\le \ell})$.

Now we show that compatible normal series of length $\ell$ can be expressed as compatible group sequences of length $\ell$ and vice versa. This illustrates that we can use these concepts interchangeably.
\begin{lemma}
    Let $L_1$ and $L_2$ be groups.
    Then $L_1$ and $L_2$ have compatible normal series of length $\ell$ if and only if
    there exists compatible group sequences 
    $((S_{i;1})_{0\le i\le\ell},\\(\pi_{i;1})_{1\le i\le\ell})$
    and $((S_{i;2})_{0\le i\le \ell},(\pi_{i;2})_{1\le i\le\ell})$ such that $L_1=S_{\ell;1}$ and $L_2=S_{\ell;2}$.
\end{lemma}

\begin{proof}
    Let $L_1$ and $L_2$ have compatible normal series
    \[1=L_{0;1}\trianglelefteq L_{1;1}\trianglelefteq\cdots\trianglelefteq L_{\ell;1}=L_1\]
    and
    \[1=L_{0;2}\trianglelefteq L_{1;2}\trianglelefteq\cdots\trianglelefteq L_{\ell;2}=L_2.\]
    For $\delta\in\{1,2\}$ and $i\in\{0,\ldots,\ell\}$, let $S_{i;\delta}:=L_{\ell;\delta}/L_{\ell-i;\delta}$.
    For $i\in\{1,\ldots,\ell\}$, let $\pi_{i;\delta}:S_{i;\delta}\rightarrow S_{i-1;\delta}$
    be the homomorphism given by $xL_{\ell-i;\delta}\mapsto xL_{\ell-i+1;\delta}$.
    Note that $\pi_{i;\delta}$ is surjective.
    For $\delta\in\{1,2\}$, let $S_\delta:=((S_{i;\delta}),(\pi_{i;\delta}))$.
    It follows that $\ker(\pi_{i;1})=L_{\ell-i+1;1}/L_{\ell-i;1}\cong L_{\ell-i+1;2}/L_{\ell-i;2}=\ker(\pi_{i;2})$.
    Hence the group sequences $S_1$ and $S_2$ are compatible.

    Conversely, let $L_1$ and $L_2$ have compatible surjective group sequences as in the statement of the lemma.
    For every $\delta\in\{1,2\}$ and $i\in\{0,\ldots,\ell\}$,
    let $L_{i;\delta}:=\ker(\pi_{\ell;\delta}\circ\pi_{\ell-1;\delta}\circ\cdots\circ\pi_{i;\delta})$.
    Note that $L_{i;\delta}$ is normal in $L_{\delta}$.
    It follows that for every $i\in\{1,\ldots,\ell\}$,
    $L_{i;1}/L_{i-1;1}\cong \ker(\pi_{\ell-i+1;1})\cong\ker(\pi_{\ell-i+1;2})\cong L_{i;2}/L_{i-1;2}.$
    Therefore, \[1=L_{0;1}\trianglelefteq L_{1;1}\trianglelefteq\cdots\trianglelefteq L_{\ell;1}=L_1\]
    and
    \[1=L_{0;2}\trianglelefteq L_{1;2}\trianglelefteq\cdots\trianglelefteq L_{\ell;2}=L_2\]
    are compatible.
    This completes the proof.
\end{proof}
We also define some operations on group sequences.
\begin{definition}
    Let 
    $S:=((S_i)_{0\le i\le\ell},(\pi_i)_{1\le i\le\ell})$
    be a group sequence.
    \begin{itemize}
        \item [\rm (i)] Let $G$ be a group and
        $f:G\rightarrow S_{\ell}$ be a homomorphism.
        Define a new group sequence $GfS$ as
        \[\xymatrix{
            G\ar[r]^f&S_{\ell}\ar[r]^{\pi_\ell}&\cdots\ar[r]^{\pi_1}&1.
        }\]
        This is called the \emph{concatenation} of $G$ and $S$ with respect to $f$.
        \item [\rm (ii)] Let $\operatorname{Contra}(S)$ be the group sequence defined as
        \[\xymatrix{
            S_\ell\ar[r]^{\pi_{\ell-1}\circ\pi_{\ell}}&S_{\ell-2}\ar[r]^{\pi_{\ell-2}}&\cdots\ar[r]^{\pi_1}&1.
        }\]
        We call $\operatorname{Contra}(S)$ the \emph{contraction} of $S$.
        For convenience, we define \[\operatorname{Contra}^2(S):=\operatorname{Contra}(\operatorname{Contra}(S)).\]
        \item [\rm (iii)] Let $T$ be a group sequence.
        Define $\operatorname{Contra}(S,T):=(\operatorname{Contra}(S),\operatorname{Contra}(T))$.
    \end{itemize}
\end{definition}
\subsection{Some technical conditions for compatibility}

\begin{definition}
    Let $G$ and $H$ be groups and let $\pi:G\rightarrow H$ be a homomorphism.
    Let $N\le H$.
    If $\pi^{-1}(M)\cong M\times\ker(\pi)$ for every $M\le N$,
    then we call $\pi$ \emph{trivially extendable} at $N$.
\end{definition}

\begin{lemma}\label{lem:goodproj}
    Let $\mathcal{X}$ be the inverse system \[\xymatrix@=0.5em{
        T_1\ar@{.}[dd]_n\ar[rd]^{\pi_1}\\
        &T_0\\
        T_n\ar[ru]_{\pi_n}
    }\] and
    let $G$ equipped with $p_i:G\rightarrow T_i$ for $i\in\{0,\ldots,n\}$ be the inverse limit of the inverse system $\mathcal{X}$.
    If $\mathcal{X}$ is surjective, then the projection $p_i$ is surjective and trivially extendable at $\ker(\pi_i)$ for every $i\in\{1,\ldots,n\}$.
\end{lemma}

\begin{proof}
    By symmetry, we only need to prove that $p_1$
    is surjective and trivially extendable at $\ker(\pi_1)$.
    Let $H:=\{(t_0,t_1,t_2,\ldots,t_n)\in T_0\times \cdots\times T_n\mid t_0=\pi_1(t_1)=\cdots=\pi_n(t_n)\}$
    and for every $i\in\{0,\ldots,n\}$, let $q_i:H\rightarrow T_i$ be the natural projection.
    By Theorem~\ref{thm:limitexists}, $H$ together with $q_i$ for $i\in\{0,\ldots,n\}$
    is an inverse limit of $\mathcal{X}$.
    Without loss of generality, we identify $G=H$ and $p_i=q_i$ for all $i\in\{0,\ldots,n\}$.
    Then 
    \[\ker(p_1)=\{(1,1,t_2,\ldots,t_n)\in T_0\times\cdots\times T_n\mid \forall\; i\in\{2,\ldots,n\}, t_i\in\ker(\pi_i)\}.\]
    Let $M$ be a subgroup of $\ker(\pi_1)$.
    Then \[p_1^{-1}(M)=\{(1,m,t_2,\ldots,t_n)\in T_0\times\cdots\times T_n\mid m\in M, \forall\; i\in\{2,\ldots,n\}, t_i\in\ker(\pi_i)\}.\]
    Let $M_1:=\{(1,m,1,\ldots,1)\in T_0\times\cdots\times T_n\mid m\in M\}$.
    Then it is clear that $p_1^{-1}(M)=\ker(p_1)M_1$.
    Also note that $M_1\cap \ker(p_1)=1$ and $[M_1,\ker(p_1)]=1$.
    Therefore, $p_1^{-1}(M)=M_1\times\ker(p_1)\cong M\times\ker(p_1)$. This completes the proof.
\end{proof}

\begin{definition}\label{def:goodwit}
    For every $\delta\in\{1,2\}$, let $L_\delta$ be a group and let 
    $N_\delta\le L_\delta$.
    If $(G,p_1,p_2)$ is a witness system for the compatibility 
    of $(L_1,L_2)$
    such that $p_\delta$ is trivially extendable at $N_\delta$ for every $\delta\in\{1,2\}$,
    then $(G,p_1,p_2)$ is called a \emph{good} witness system
    associated with $(N_1,N_2)$.
\end{definition}

\begin{example}\label{examp:goodwit}
  Let 
  \[\begin{aligned}
    L_1 &:= \mathbb{Z}_p^n 
          = \langle x_1\rangle \times \cdots \times \langle x_n\rangle,
    \\
    L_2 &:= \mathbb{Z}_{p^n}
          = \langle y\rangle,\\
      G &:= \mathbb{Z}_p \times \cdots \times \mathbb{Z}_{p^n}
         = \langle a_1\rangle \times \cdots \times \langle a_n\rangle.
  \end{aligned}  
  \]
  Let 
  \[
    p_1: G \longrightarrow L_1,
    \qquad
    a_i \mapsto x_i
    \quad(\forall\,i=1,\dots,n),
  \]
  and let
  \[
    p_2: G \longrightarrow L_2,
    \qquad
    a_i \mapsto
    \begin{cases}
      1, & i=1,\dots,n-1,\\
      y, & i=n.
    \end{cases}
  \]
  Then \((G,p_1,p_2)\) is a good witness system associated with $\langle x_1\rangle$ and $\langle y^{p^{n-1}}\rangle$.
  
  Let $L_1':=\langle x_2\rangle\times\cdots\times\langle x_n\rangle\cong\mathbb{Z}_p^{n-1}$, 
  let $L_2':=\langle y^{p}\rangle\cong\mathbb{Z}_{p^{n-1}}$,
  let $\pi_1:L_1\rightarrow L_1'$ be the natural projection
  and let $\pi_2:L_2\rightarrow L_2'$ be given by $y\mapsto y^p$.
  Then $\ker(\pi_1)=\langle x_1\rangle$ and $\ker(\pi_2)=\langle y^{p^{n-1}}\rangle$, and we have that
\((G,\pi_1\circ p_1,\pi_2\circ p_2)\) is a witness system for the compatibility of $(L_1',L_2')$.
  This is an important observation, as it shows that the good witness system
  \((G,p_1,p_2)\) can be used to construct witness systems for quotients of $L_1$ and $L_2$.
  In fact, this is the key idea of our method (see Lemma~\ref{lem:induction}). 
\end{example}

For the remainder of this section, we always assume $\delta\in\{1,2\}$. 
We also make the convention
that when we mention something indexed by $\delta$, it means we are discussing 
it for every $\delta\in\{1,2\}$. 
For example, in the above definition, 
the sentence `$p_\delta$ is trivially extendable at $N_\delta$ for every $\delta\in\{1,2\}$'
can be written as `$p_\delta$ is trivially extendable at $N_\delta$'.

\begin{lemma}\label{lem:induction}
    Let $S_\delta:=((S_{i;\delta})_{0\le i\le 3},(\pi_{i;\delta})_{1\le i\le 3})$ 
    be surjective group sequences such that
    \begin{itemize}
        \item [\rm (i)] $\ker(\pi_{3;1})\cong\ker(\pi_{3;2})$, and
        \item [\rm (ii)] $\pi_{3;\delta}$ is trivially extendable at $\ker(\pi_{2;\delta})$.
    \end{itemize}
    If there exists a good witness system $(G,p_1,p_2)$ for the 
    compatibility of $(S_{3;1},S_{3;2})$
    associated with $(\ker(\pi_{2;1}\circ\pi_{3;1}), 
    \ker(\pi_{2;2}\circ\pi_{3;2}))$, then
    $(G,\pi_{3;1}\circ p_1,\pi_{3;2}\circ p_2)$ is a good witness system for the compatibility of
    $(S_{2;1},S_{2;2})$ 
    associated with $(\ker(\pi_{2;1}),\ker(\pi_{2;2}))$.
\end{lemma}

\begin{proof}
    Let $q_\delta:=\pi_{3;\delta}\circ p_\delta$. 
    It can be seen that $q_\delta$ is surjective.
    We first show that $(G,q_1,q_2)$ is a witness system for $(S_{2;1},S_{2;2})$.
    Notice that $p_\delta$ is surjective and $\ker(q_\delta)=p_\delta^{-1}(\ker(\pi_{3;\delta}))$.
    As $\ker(\pi_{3;\delta})\le\ker(\pi_{2;\delta}\circ\pi_{3;\delta})$ and
    $p_\delta$ is trivially extendable at $\ker(\pi_{2;\delta}\circ\pi_{3;\delta})$,
    we have 
    \[\ker(q_\delta)\cong \ker(\pi_{3;\delta})\times\ker(p_\delta).\]
    Hence $\ker(q_1)\cong\ker(q_2)$ by {\rm (i)}.
    Now let $M_\delta\le\ker(\pi_{2;\delta})$. 
    Since 
    $\pi^{-1}_{3;\delta}(M_\delta)\le\ker(\pi_{2;\delta}\circ\pi_{3;\delta})$,
    we have
    \[\begin{aligned}
        q_\delta^{-1}(M_\delta)&=p_\delta^{-1}(\pi^{-1}_{3;\delta}(M_\delta))\\
        &\cong \pi^{-1}_{3;\delta}(M_\delta)\times\ker(p_\delta)\\
        &\overset{\rm (ii)}{\cong} M_\delta\times\ker(\pi_{3;\delta})\times\ker(p_\delta)\\
        &\cong M_\delta\times\ker(q_\delta).
    \end{aligned}\]
    This completes the proof.
\end{proof}

Now we rephrases Example~\ref{examp:goodwit} in terms of Lemma~\ref{lem:induction}.
\begin{example}
    Using the notation in Example~\ref{examp:goodwit},
    let $S_{3;\delta}:=L_\delta$, 
    let $S_{2;\delta}:=L'_\delta$ and let $S_{1;\delta}=S_{2;\delta}$.
    Let $\pi_{3;\delta}=\pi_\delta$ and let $\pi_{2;\delta}={\rm id}_{L'_\delta}$.
    Since $\ker(\pi_{2;\delta})=1$,
    we have that $(G,p_1,p_2)$ is a good witness system for the compatibility of $(S_{3;1},S_{3;2})$
    associated with $(\ker(\pi_{2;1}\circ\pi_{3;1}), \ker(\pi_{2;2}\circ\pi_{3;2}))$
    and $\pi_{3;\delta}$ is trivially extendable at $\ker(\pi_{2;\delta})$.
    By Lemma~\ref{lem:induction}, $(G,\pi_{3;1}\circ p_1,\pi_{3;2}\circ p_2)$ is a witness system for the compatibility of $(S_{2;1},S_{2;2})$.
\end{example}

Lemma~\ref{lem:induction} motivates us to define the following concept.
\begin{definition}
    Let  $S_\delta:=((S_{i;\delta})_{0\le i\le \ell},(\pi_{i;\delta})_{1\le i\le\ell})$
    % and
    % \[S_2:=((S_{i;2})_{0\le i\le\ell},(\pi_{i;2})_{1\le i\le \ell})\] 
    be a surjective group sequence.
    If $(G,p_1,p_2)$ is a good witness system for the compatibility of $(S_{\ell;1},S_{\ell;2})$
    associated with $(\ker(\pi_{\ell;1}),\ker(\pi_{\ell;2}))$,
    then $(G,p_1,p_2)$ is called a \emph{good witness system} associated with $(S_1,S_2)$.
\end{definition}

\begin{lemma}\label{lem:length2}
    If $S_1$ and $S_2$ are compatible surjective group sequences of length $2$,
    then
    there exists a good witness system associated with $(S_1,S_2)$.
\end{lemma}
\begin{proof}
    Since $S_1$ and $S_2$ are compatible, there exists
    isomorphism $\sigma:\ker(\pi_{1;1})\rightarrow\ker(\pi_{1;2})$.
    Since $S_{0;\delta}=1$, we have $\ker(\pi_{1;\delta})=S_{1;\delta}$.
    Let $\mathcal{X}$ be the inverse system given by
    \[\xymatrix@=0.5em{
        S_{2;1}\ar[rd]^{\sigma\circ\pi_{2;1}}&\\
        &S_{1;2}\\
        S_{2;2}\ar[ru]_{\pi_{2;2}}
    }\]
    and let $G$ equipped with $p_1:G\rightarrow S_{2;1}$, $p_2:G\rightarrow S_{2;2}$ and
    $p:G\rightarrow S_{2;1}$ be the inverse limit of $\mathcal{X}$.
    Let $\mathcal{Y}_1$ be the inverse system given by
    \[\xymatrix@=0.5em{
        S_{2;1}\ar[rd]^{\sigma\circ\pi_{2;1}}&\\
        &S_{1;2}\\
        S_{1;2}\ar@{=}[ru]
    }\]
    By Theorem \ref{thm:projection} {\rm (ii)},
    $p_1$ is the inverse limit of the morphism from $\mathcal{X}$ to $\mathcal{Y}_1$
    defined as below:
    \[\xymatrix@=0.5em{
        S_{2;1}\ar[rd]^{\sigma\circ\pi_{2;1}}\ar@{=}[rrrr]&&&&S_{2;1}\ar[rd]^{\sigma\circ\pi_{2;1}}\\
        &S_{1;2}\ar@{=}[rrrr]&&&&S_{1;2}\\
        S_{2;2}\ar[ru]_{\pi_{2;2}}\ar[rrrr]_{\pi_{2;2}}&&&&S_{1;2}\ar@{=}[ru]
    }\]
    Then by Theorem \ref{thm:limpullback} {\rm (iii)},
    $\ker(p_1)\cong \ker(\pi_{2;2})$.
    Similarly, we deduce that
    $\ker(p_2)\cong \ker(\pi_{2;1})$.
    Therefore,
    $\ker(p_1)\cong\ker(p_2)$.
    It follows that $(G,p_1,p_2)$ is a witness system
    for the compatibility of $(S_{2;1},S_{2;2})$.
    By Lemma \ref{lem:goodproj}, 
    $p_1$ is surjective and trivially extendable
    at $\ker(\sigma\circ\pi_{2;1})=\ker(\pi_{2;1})$.
    We also have that $p_2$ is surjective and trivially
    extendable at $\ker(\pi_{2;2})$.
    Therefore, $(G,p_1,p_2)$ is a good witness system
    for the compatibility of $(S_{2;1},S_{2;2})$
    associated with $(\ker(\pi_{2;1}),\ker(\pi_{2;2}))$.
    This concludes the proof.
\end{proof}

By applying Lemma \ref{lem:induction}, we obtain the following lemma.

\begin{lemma}\label{lem:induction1}
    Let $\ell\ge 3$ and let $S_\delta:=((S_{i;\delta})_{0\le i\le \ell},(\pi_{i;\delta})_{1\le i\le\ell})$
    be surjective group sequences of length $\ell$.    
    Let $\mathcal{C}_{\ell-1}$ be a collection of pairs of surjective group sequences
    of length $\ell-1$.
    If
    \begin{itemize}
        \item [\rm (i)] for every $(T_1,T_2)\in \mathcal{C}_{\ell-1}$, there is a good witness system associated with
        $(T_1,T_2)$, and
        \item [\rm (ii)] there exists a group $S_{\ell+1;\delta}$ and a homomorphism $\pi_{\ell+1;\delta}:S_{\ell+1;\delta}\rightarrow S_{\ell;\delta}$ 
        such that
        \begin{itemize}
            \item [\rm (a)] $\ker(\pi_{\ell+1;1})\cong\ker(\pi_{\ell+1;2})$,
            \item [\rm (b)] $\pi_{\ell+1;\delta}$ is trivially extendable at $\ker(\pi_{\ell;\delta})$, and
            \item [\rm (c)] $\operatorname{Contra}^2(S_{\ell+1;1}\pi_{\ell+1;1}S_{1},S_{\ell+1;2}\pi_{\ell+1;2}S_{2})\in\mathcal{C}_{\ell-1}$,
        \end{itemize}
    \end{itemize}
    then there exists a good witness system associated with $(S_1,S_2)$.

\end{lemma}

\begin{proof}
    By {\rm (i)} and {\rm (ii.c)}, there exists a good witness system
    $(G,p_1:G\rightarrow S_{\ell+1;1},p_2:G\rightarrow S_{\ell+1;2})$ 
    associated with 
    $(\ker(\pi_{\ell-1;1}\circ\pi_{\ell;1}\circ\pi_{\ell+1;1}),
    \ker(\pi_{\ell-1;2}\circ\pi_{\ell;2}\circ\pi_{\ell+1;2}))$.
    This implies that 
    $(G,p_1,p_2)$ is good associated with
    $(\ker(\pi_{\ell;1}\circ\pi_{\ell+1;1}),\ker(\pi_{\ell;2}\circ\pi_{\ell+1;2}))$.

    By {\rm (ii.a)} and {\rm (ii.b)}, we can apply Lemma \ref{lem:induction} to
    \[\xymatrix{T_1:S_{\ell+1;1}\ar[r]^{\pi_{\ell+1;1}}&S_{\ell;1}\ar[r]^{\pi_{\ell;1}}&S_{\ell-1;1}\ar[r]&1}\]
    and
    \[\xymatrix{T_2:S_{\ell+1;2}\ar[r]^{\pi_{\ell+1;2}}&S_{\ell;2}\ar[r]^{\pi_{\ell;2}}&S_{\ell-1;2}\ar[r]&1}.\]
    Hence we deduce that $(G,\pi_{\ell+1;1}\circ p_1,\pi_{\ell+1;2}\circ p_2)$ 
    is a good witness system for the compatibility of $(S_{\ell;1},S_{\ell;2})$
    associated with $(\ker(\pi_{\ell;1}),\ker(\pi_{\ell;2}))$, which also implies that
    this is a good witness system associated with $(S_1,S_2)$.
\end{proof}

\begin{definition}
    Let $S:=((S_i)_{0\le i\le \ell},(\pi_i)_{1\le i\le\ell})$ and $T:=((T_i)_{0\le i\le \ell},(\rho_i)_{1\le i\le\ell})$ 
    be surjective group sequences.
    \begin{itemize}
        \item [\rm (i)] If  $S_i=T_i$ and $\pi_i=\rho_i$ for every $1\le i\le \ell-1$, 
        then we say that $S$ and $T$ are
        \emph{almost equal} and denote this by $S\approx T$.
        \item [\rm (ii)] Let $S\approx T$. 
        Let \[\xymatrix@=0.5em{
            &S_\ell\ar[rd]^{\pi_\ell}&\\
            S'_{\ell}:=\varprojlim&&S_{\ell-1}\\
            &T_\ell\ar[ru]_{\rho_\ell}
        }\]
        and let $\pi'_\ell:S'_\ell\rightarrow S_{\ell-1}$ be the inverse limit projection from
        $S'_\ell$ to $S_{\ell-1}$.
        We define a new group sequence
        \[\xymatrix{
            S\# T:S'_\ell\ar[r]^{\pi'_\ell}&S_{\ell-1}\ar[r]^{\pi_{\ell-1}}&\cdots\ar[r]^{\pi_{2}}&S_1\ar[r]^{\pi_1}&1}
        \]
        \item [\rm (iii)] Let $\mathcal{C}$ be a collection of pairs of group sequences of length $\ell$.
        If for every $(S_1,S_2),(T_1,T_2)\in\mathcal{C}$ such that $S_\delta\approx T_\delta$,
        we have \[(S_1\# T_1,S_2\# T_2)\in\mathcal{C},\] then we say that $\mathcal{C}$
        is \emph{closed} under $\#$.
        
    \end{itemize}
    
\end{definition}

\begin{definition}
    Let $S_\delta:=((S_{i;\delta})_{0\le i\le\ell},(\pi_{i;\delta})_{1\le i\le\ell})$
    be surjective group sequences and let $(S_1,S_2)$ be compatible.
    \begin{itemize}
        \item [\rm (i)] If $\operatorname{Contra}(S_1,S_2)$ is compatible,
        we call $(S_1,S_2)$ \emph{reducible}.
        \item [\rm (ii)] Let $\mathcal{C}_\ell$ be a collection of pairs of compatible surjective group sequences of length $\ell$
        and let $\mathcal{C}_{\ell-1}$ be a collection of pairs of compatible surjective group sequences of length $\ell-1$.
        If for every reducible $(S_1,S_2)\in\mathcal{C}_\ell$, we have $\operatorname{Contra}(S_1,S_2)\in\mathcal{C}_{\ell-1}$,
        then we say that $\mathcal{C}_\ell$ \emph{can be reduced to} $\mathcal{C}_{\ell-1}$ and denote this by $\mathcal{C}_{\ell}\searrow\mathcal{C}_{\ell-1}$.
    \end{itemize}
    
\end{definition}

\begin{lemma}
    \label{lem:induction3}
    Let $\ell\ge 3$, let $\mathcal{C}_\ell$ be a collection of pairs of compatible group sequences of length $\ell$ which is closed under $\#$,
    and let $\mathcal{C}_{\ell-1}$ be a collection of pairs of compatible surjective group sequences
    of length $\ell-1$ such that $\mathcal{C}_\ell\searrow\mathcal{C}_{\ell-1}$.
    Recall that $\delta$ runs through $\{1,2\}$. 
    Let $\bar{\delta}:=3-\delta\in\{1,2\}$.
    Let $(S_1,S_2)\in\mathcal{C}_\ell$ with $S_\delta:=((S_{i;\delta})_{0\le i\le \ell},(\pi_{i;\delta})_{1\le i\le\ell})$.
    If
    \begin{itemize}
        \item [\rm (i)] for every $(T_1,T_2)\in \mathcal{C}_{\ell-1}$, there is a good witness system associated with
        $(T_1,T_2)$;
        \item [\rm (ii)] there exist a group $G_\delta$ and a surjective homomorphism $\rho_{\delta}:G_\delta\rightarrow S_{\ell;\delta}$ 
        such that $\rho_{\delta}$ is trivially extendable at $\ker(\pi_{\ell;\delta})$;
        \item [\rm (iii)] there exist a group $H_\delta$, a homomorphism $\phi_{\delta}:H_\delta\rightarrow S_{\ell-1;\delta}$ 
        and isomorphisms $\sigma_{\ell-1}:\ker(\pi_{\ell-1;1})\rightarrow \ker(\pi_{\ell-1;2}$), 
        $\eta_\delta:\ker(\pi_{\ell-1;\delta}\circ\phi_\delta)\rightarrow \ker(\pi_{\ell-1;\bar{\delta}}\circ\pi_{\ell;\bar{\delta}}\circ \rho_{\bar{\delta}})$
        such that the following diagrams commute 
        \[\xymatrix{
            \ker(\pi_{\ell-1;1}\circ\phi_1)\ar[rr]^{\phi_1}\ar[d]^{\eta_1}&&\ker(\pi_{\ell-1;1})\ar[d]^{\sigma_{\ell-1}}\\
            \ker(\pi_{\ell-1;2}\circ\pi_{\ell;2}\circ\rho_2)\ar[rr]^{\pi_{\ell;2}\circ\rho_2}&&\ker(\pi_{\ell-1;2})
        }\]
        \[\xymatrix{
            \ker(\pi_{\ell-1;2}\circ\phi_2)\ar[rr]^{\phi_2}\ar[d]^{\eta_2}&&\ker(\pi_{\ell-1;2})\ar[d]^{\sigma^{-1}_{\ell-1}}\\
            \ker(\pi_{\ell-1;1}\circ\pi_{\ell;1}\circ\rho_1)\ar[rr]^{\pi_{\ell;1}\circ\rho_1}&&\ker(\pi_{\ell-1;1})
        }\]
        \item [\rm (iv)] $(S_1',S_2'), (S_1'',S_2'')\in\mathcal{C}_\ell$ where
        \[\xymatrix{
        S_1':G_1\ar[r]^{\pi_{\ell;1}\circ\rho_1}&S_{\ell-1;1}\ar[r]^{\pi_{\ell-1;1}}&\cdots\ar[r]^{\pi_{1;1}}&1
    }\]
    \[\xymatrix{S_1'':H_1\ar[r]^{\phi_1}&S_{\ell-1;1}\ar[r]^{\pi_{\ell-1;1}}&\cdots\ar[r]^{\pi_{1;1}}&1}\]
    \[\xymatrix{S_2':H_2\ar[r]^{\phi_2}&S_{\ell-1;2}\ar[r]^{\pi_{\ell-1;2}}&\cdots\ar[r]^{\pi_{1;2}}&1}\]
    \[\xymatrix{
        S_2'':G_2\ar[r]^{\pi_{\ell;2}\circ\rho_2}&S_{\ell-1;2}\ar[r]^{\pi_{\ell-1;2}}&\cdots\ar[r]^{\pi_{1;2}}&1
    }\]
    \end{itemize}
    then there exists a good witness system associated with $(S_1,S_2)$.
\end{lemma}

\begin{proof}
    Let $\mathcal{X}_\delta$ be the inverse system
    \[\xymatrix@=0.5em{
        G_\delta\ar[rd]^{\pi_{\ell;\delta}\circ\rho_\delta}&\\
        &S_{\ell-1;\delta}\\
        H_\delta\ar[ru]_{\phi_\delta}
    }\] 
    and let $S_{\ell+1;\delta}:=\varprojlim\mathcal{X}_\delta$.
    Let $\mathcal{Y}_\delta$ be the inverse system
    \[\xymatrix@=0.5em{
        S_{\ell;\delta}\ar[rd]^{\pi_{\ell;\delta}}\\
        &S_{\ell-1;\delta}\\
        S_{\ell-1;\delta}\ar@{=}[ru]
    }\]
    Notice that $S_{\ell;\delta}=\varprojlim(\mathcal{Y}_\delta)$.    
    We let $\pi_{\ell+1;\delta}:S_{\ell+1;\delta}\rightarrow S_{\ell;\delta}$ be the 
    inverse limit of the morphism $X_\delta\rightarrow Y_\delta$ defined as follows
    \[\xymatrix@=0.5em{
        G_\delta\ar[rd]_{\pi_{\ell;\delta}\circ\rho_\delta}\ar[rr]^{\rho_\delta}&&S_{\ell;\delta}\ar[rd]^{\pi_{\ell;\delta}}\\
        &S_{\ell-1;\delta}\ar@{=}[rr]&&S_{\ell-1;\delta}\\
        H_\delta\ar[ru]^{\phi_\delta}\ar[rr]^{\phi_\delta}&&S_{\ell-1;\delta}\ar@{=}[ru]
    }\]
    
    Now we compute $\ker(\pi_{\ell+1;\delta})$.
    By Theorem \ref{thm:limpullback} {\rm (iii)}, $\ker(\pi_{\ell+1;\delta})$
    is the inverse limit of
    \[\xymatrix@=0.5em{
        \ker(\rho_\delta)\ar[rd]\\
        &1\\
        \ker(\phi_\delta)\ar[ru]
    }\]
    which implies that $\ker(\pi_{\ell+1;\delta})\cong\ker(\rho_\delta)\times\ker(\phi_\delta)$.
    By {\rm (iii)}, $\ker(\phi_1)\cong \ker(\pi_{\ell;2}\circ\rho_2)$
    and $\ker(\phi_2)\cong \ker(\pi_{\ell;1}\circ\rho_1)$.
    By {\rm (ii)}, 
    \[\ker(\pi_{\ell;\delta}\circ\rho_\delta)=\rho_\delta^{-1}(\ker(\pi_{\ell;\delta}))\cong \ker(\pi_{\ell;\delta})\times\ker(\rho_\delta).\]
    Also notice that $\ker(\pi_{\ell;1})\cong\ker(\pi_{\ell;2})$ by the
    compatibility of $S_1$ and $S_2$. Thus, we have 
    \[\begin{aligned}
        \ker(\pi_{\ell+1;1})&\cong \ker(\rho_1)\times\ker(\phi_1)\\
        &\cong \ker(\rho_1)\times\ker(\pi_{\ell;2}\circ\rho_1)\\
        &\cong \ker(\rho_1)\times\ker(\pi_{\ell;2})\times \ker(\rho_2)\\
        &\cong \ker(\rho_1)\times\ker(\pi_{\ell;1})\times\ker(\rho_2)\\
        &\cong \ker(\pi_{\ell;1}\circ\rho_1)\times\ker(\rho_2)\\
        &\cong \ker(\phi_2)\times \ker(\rho_2)\\
        &\cong \ker(\pi_{\ell+1;2})
    \end{aligned}
    \]
    This implies that $S_{\ell+1;\delta}$ and $\pi_{\ell+1;\delta}$
    satisfy Lemma \ref{lem:induction1} {\rm (ii.a)}.
    Let $M_\delta\le \ker(\pi_{\ell;\delta})$.
    We now compute $\pi_{\ell+1;\delta}^{-1}(M_\delta)$.
    Note that $M_\delta$ is the inverse limit
    of
    \[\xymatrix@=0.5em{
        M_\delta\ar[rd]\\
        &1\\
        1\ar[ru]
    }\]
    which is a subsystem of $\mathcal{Y}_\delta$.
    Then by Theorem \ref{thm:limpullback} {\rm (ii)}, $\pi^{-1}_{\ell+1;\delta}(M_\delta)$
    is the inverse limit of 
    \[\xymatrix@=0.5em{
        \rho_\delta^{-1}(M_\delta)\ar[rd]\\
        &1\\
        \ker(\phi_\delta)\ar[ru]
    }\]
    which implies that
    $\pi^{-1}_{\ell+1;\delta}(M_\delta)\cong\rho^{-1}_{\delta}(M_\delta)\times\ker(\phi_\delta)$.
    By {\rm (ii)}, 
    \[\pi^{-1}_{\ell+1;\delta}(M_\delta)\cong M_\delta\times \ker(\rho_\delta)\times\ker(\phi_\delta)\cong M_\delta\times\ker(\pi_{\ell+1;\delta}).\]
    We therefore derive that $S_{\ell+1;\delta}$ and $\pi_{\ell+1;\delta}$
    satisfy Lemma \ref{lem:induction1} {\rm (ii.b)}.
    
    Now we turn to condition {\rm (ii.c)} in Lemma \ref{lem:induction1}. 
    Let 
    \[\xymatrix{
        T_1: S_{\ell+1;1}\ar[r]^{\pi_{\ell;1}\circ\pi_{\ell+1;1}}&S_{\ell-1;1}\ar[r]^{\pi_{\ell-1;1}}&\cdots\ar[r]&1
    }\]
    and let
    \[\xymatrix{
        T_2: S_{\ell+1;2}\ar[r]^{\pi_{\ell;2}\circ\pi_{\ell+1;2}}&S_{\ell-1;2}\ar[r]^{\pi_{\ell-1;2}}&\cdots\ar[r]&1.
    }\]
    Then $(T_1,T_2)=\operatorname{Contra}(S_{\ell+1;1}\pi_{\ell+1;1}S_1,S_{\ell+1;2}\pi_{\ell+1;2}S_2)$.
    By our previous discussion, 
    \[
        \ker(\pi_{\ell;1}\circ\pi_{\ell+1;1})\cong\ker(\pi_{\ell;1})\times\ker(\pi_{\ell+1;1})\cong\ker(\pi_{\ell;2})\times\ker(\pi_{\ell+1;2})\cong\ker(\pi_{\ell;2}\circ\pi_{\ell+1;2})
    \]
    which implies that $T_1$ and $T_2$ are compatible.
    Let $\mathcal{Z}_\delta$ be the inverse system
    \[\xymatrix@=0.5em{
        S_{\ell-1;\delta}\ar@{=}[rd]\\
        &S_{\ell-1;\delta}\\
        S_{\ell-1;\delta}\ar@{=}[ru]
    }\]
    Note that $S_{\ell-1;\delta}=\varprojlim(\mathcal{Z}_\delta)$ and
    $\pi_{\ell;\delta}:S_{\ell;\delta}\rightarrow S_{\ell-1;\delta}$ is the inverse limit of the morphism $\mathcal{Y}_\delta\rightarrow\mathcal{Z}_\delta$
    defined as
    \[\xymatrix@=0.5em{
        S_{\ell;\delta}\ar[rd]_{\pi_{\ell;\delta}}\ar[rrrr]^{\pi_{\ell;\delta}}&&&&S_{\ell-1;\delta}\ar@{=}[rd]\\
        &S_{\ell-1;\delta}\ar@{=}[rrrr]&&&&S_{\ell-1;\delta}\\
        S_{\ell-1;\delta}\ar@{=}[ru]\ar@{=}[rrrr]&&&&S_{\ell-1;\delta}\ar@{=}[ru]
    }\]
    We then derive that
    $\pi_{\ell;\delta}\circ\pi_{\ell+1;\delta}:S_{\ell+1;\delta}\rightarrow S_{\ell-1;\delta}$ is the inverse limit of the morphism 
    $\Pi_\delta:\mathcal{X}_\delta\rightarrow\mathcal{Z}_\delta$
    defined as
    \[\xymatrix@=0.5em{
        G_\delta\ar[rd]_{\pi_{\ell;\delta}\circ\rho_\delta}\ar[rrrr]^{\pi_{\ell;\delta}\circ\rho_\delta}&&&&S_{\ell-1;\delta}\ar@{=}[rd]\\
        &S_{\ell-1;\delta}\ar@{=}[rrrr]&&&&S_{\ell-1;\delta}\\
        H_\delta\ar[ru]^{\phi_\delta}\ar[rrrr]^{\phi_\delta}&&&&S_{\ell-1;\delta}\ar@{=}[ru]
    }\]
    Let $\mathcal{W}_\delta$ be the inverse system
    \[\xymatrix@=0.5em{
        S_{\ell-2;\delta}\ar@{=}[rd]\\
        &S_{\ell-2;\delta}\\
        S_{\ell-2;\delta}\ar@{=}[ru]
    }\]
    It can be seen that $\pi_{\ell-1;\delta}$ is the inverse limit of the
    morphism $\mathcal{Z}_\delta\rightarrow\mathcal{W}_\delta$ defined as
    \[\xymatrix@=0.5em{
        S_{\ell-1;\delta}\ar@{=}[rd]\ar[rrrr]^{\pi_{\ell-1;\delta}}&&&&S_{\ell-2;\delta}\ar@{=}[rd]\\
        &S_{\ell-1;\delta}\ar[rrrr]^{\pi_{\ell-1;\delta}}&&&&S_{\ell-2;\delta}\\
        S_{\ell-1;\delta}\ar@{=}[ru]\ar[rrrr]^{\pi_{\ell-1;\delta}}&&&&S_{\ell-2;\delta}\ar@{=}[ru]
    }\]
    Hence $\pi_{\ell-1;\delta}\circ\pi_{\ell;\delta}\circ\pi_{\ell+1;\delta}:S_{\ell+1;\delta}\rightarrow S_{\ell-1;\delta}$
    is the inverse limit of the morphism $\mathcal{X}_\delta\rightarrow\mathcal{W}_\delta$
    defined as
    \[\xymatrix@=0.5em{
        G_\delta\ar[rd]_{\pi_{\ell;\delta}\circ\rho_\delta}\ar[rrrr]^{\pi_{\ell-1;\delta}\circ\pi_{\ell;\delta}\circ\rho_\delta}&&&&S_{\ell-2;\delta}\ar@{=}[rd]\\
        &S_{\ell-1;\delta}\ar[rrrr]^{\pi_{\ell-1;\delta}}&&&&S_{\ell-2;\delta}\\
        H_\delta\ar[ru]^{\phi_\delta}\ar[rrrr]^{\pi_{\ell-1;\delta}\circ\phi_\delta}&&&&S_{\ell-2;\delta}\ar@{=}[ru]
    }\]
    By Corollary \ref{thm:limpullback}, $\ker(\pi_{\ell-1;\delta}\circ\pi_{\ell;\delta}\circ\pi_{\ell+1;\delta})$
    is the inverse limit of the inverse system $\mathcal{K}_\delta$ defined as
    \[\xymatrix@=0.5em{
        \ker(\pi_{\ell-1;\delta}\circ\pi_{\ell;\delta}\circ\rho_\delta)\ar[rd]_{\pi_{\ell;\delta}\circ\rho_\delta}\\
        &\ker(\pi_{\ell-1;\delta})\\
        \ker(\pi_{\ell-1;\delta}\circ\phi_\delta)\ar[ru]^{\phi_\delta}
    }\]
    Then by {\rm (iii)}, there is an isomorphism from $\ker(\pi_{\ell-1;1}\circ\pi_{\ell;1}\circ\pi_{\ell+1;1})$
    to $\ker(\pi_{\ell-1;2}\circ\pi_{\ell;2}\circ\pi_{\ell+1;2})$
    defined as the inverse limit of the following morphism $\mathcal{K}_1\rightarrow\mathcal{K}_2$:
    \[\xymatrix@=0.5em{
        \ker(\pi_{\ell-1;1}\circ\pi_{\ell;1}\circ\rho_1)\ar[rd]_{\pi_{\ell;1}\circ\rho_1}\ar[rr]^{\eta_2^{-1}}&&\ker(\pi_{\ell-1;2}\circ\phi_2)\ar[rd]^{\phi_2}\\
        &\ker(\pi_{\ell-1;1})\ar[rr]^{\sigma_{\ell-1}}&&\ker(\pi_{\ell-1;2})\\
        \ker(\pi_{\ell-1;1}\circ\phi_1)\ar[ru]^{\phi_1}\ar[rr]^{\eta_1}&&\ker(\pi_{\ell-1;2}\circ\pi_{\ell;2}\circ\rho_2)\ar[ru]_{\ \ \ \ \ \ \pi_{\ell;2}\circ\rho_2}
    }\]
    Therefore,
    \[\ker(\pi_{\ell-1;1}\circ\pi_{\ell;1}\circ\pi_{\ell+1;1})\cong\ker(\pi_{\ell-1;2}\circ\pi_{\ell;2}\circ\pi_{\ell+1;2})\]
    which implies that $(T_1,T_2)$ is reducible.
    Let $\pi_\delta$ be the limit projection
    from $S_{\ell+1;\delta}$ to $S_{\ell-1;\delta}$.
    By Theorem \ref{thm:projection}, $\pi_\delta:S_{\ell+1;\delta}\rightarrow S_{\ell-1;\delta}$
    is also the inverse limit of $\Pi_\delta$.
    Then by the definition of inverse limits, $\pi_\delta=\pi_{\ell;\delta}\circ\pi_{\ell+1;\delta}$,
    which implies that $S'_1\# S''_1=T_1$ and $S'_2\#S''_2=T_2$.
    Since $\mathcal{C}_\ell$ is closed under $\#$, it follows that $(T_1,T_2)\in\mathcal{C}_\ell$ by {\rm (iv)}. 
    Also recall that $\mathcal{C}_\ell\searrow\mathcal{C}_{\ell-1}$.
    We have \[\operatorname{Contra}^2(S_{\ell+1;1}\pi_{\ell+1;1}S_1,S_{\ell+1;2}\pi_{\ell+1;2}S_2)=\operatorname{Contra}(T_1,T_2)\in\mathcal{C}_{\ell-1}.\]
    Then Lemma \ref{lem:induction1} {\rm (ii.c)} is satisfied.
    This completes the proof.
\end{proof}

Lemma \ref{lem:induction3} indicates that if we could construct groups
$G_\delta$ and $H_\delta$ satisfying {\rm (ii)}, {\rm (iii)}, {\rm (iv)} 
for every $\ell$,
then we could prove inductively that if two groups have compatible group sequences then
they have a good witness system.
So far, we can only construct such an operator for compatible
sequences satisfying certain extra hypotheses.

\subsection{Proofs of main results}
\begin{definition}
    Let $G_1$ and $G_2$ be groups. 
    Let $f:G_1\rightarrow G_2$ be a homomorphism.
    If $f$ is surjective, a \emph{transversal} with respect to $f$
    is a map $\tau:G_2\rightarrow G_1$ such that $f\circ \tau=\mathrm{id}_{G_2}$.
    If $f$ is an isomorphism,
    then we define $f_\bullet$ as follows:
    $f_\bullet:\Aut(G_1)\rightarrow \Aut(G_2),\sigma\mapsto f\circ\sigma\circ f^{-1}$.
    Note that $f_\bullet$ is an isomorphism.

\end{definition}

\begin{definition}
    Let $G$ be a group and let $H$ be a subgroup of $G$.
    Let $\sigma$ be an automorphism of $G$. 
    If $\sigma(H)=H$, then
    we denote $\sigma|_{H,H}$ as $\sigma^H$.
    Let $A\le \Aut(G)$. Define $A_H:=\{\sigma\in A\mid \sigma(H)=H\}$ 
    and $A^H:=\{\sigma^H\in\Aut(H)\mid \sigma\in A_H\}$.
\end{definition}

\begin{definition}\label{def:Comp}
    Let $\mathbf{Comp}_2$ be the collection of all pairs of compatible group sequences.
    For $\ell\ge 3$, let $\mathbf{Comp}_\ell$ be the collection of all pairs
    of compatible group sequences $(((S_{i;1}),(\pi_{i;1})),((S_{i;2}),(\pi_{i;2})))$ of length $\ell$ 
    satisfying the following condition:
    for every $i\in\{2,\ldots,\ell-1\}$, there exists an isomorphism 
    $\sigma_i:\ker(\pi_{i;1})\rightarrow \ker(\pi_{i;2})$ such that
    \begin{equation}\label{eqn:addrestriction}(\sigma_i^{\bar{\delta}-\delta})_\bullet(\Inn(S_{i;\delta})^{\ker(\pi_{i;\delta})})\le \Aut(S_{i;\bar{\delta}})^{\ker(\pi_{i;\bar{\delta}})}\end{equation}
    where $\bar{\delta}=3-\delta$.
\end{definition}

\begin{lemma}\label{lem:comp_l}
    Let $\ell\ge 3$ and let $(S_1,S_2):=(((S_{i;1}),(\pi_{i;1})),((S_{i;2}),(\pi_{i;2})))$
    be a pair of compatible group sequences of length $\ell$. 
    The following two statements are equivalent.
    \begin{itemize}
        \item [\rm (i)] $(S_1,S_2)\in \mathbf{Comp}_\ell$.
        \item [\rm (ii)] Let $\bar{\delta}=3-\delta$. For every $ i\in\{2,\ldots,\ell-1\}$, there exist an isomorphism $\sigma_i:\ker(\pi_{i;1})\rightarrow \ker(\pi_{i;2})$, 
        a map $\alpha_{i;\bar{\delta}}:S_{i-1;\delta}\rightarrow \Aut(S_{i;\bar{\delta}})_{\ker(\pi_{i;\bar{\delta}})}$, 
        and a transversal $\tau_{i;\delta}:S_{i-1;\delta}\rightarrow S_{i;\delta}$ with respect to $\pi_{i;\delta}$
        such that 
        \[(\sigma_{i}^{\bar{\delta}-\delta})_\bullet(\Inn(\tau_{i;\delta}(x)^{-1})^{\ker(\pi_{i;\delta})})=\alpha_{i;\bar{\delta}}(x)^{\ker(\pi_{i;\bar{\delta}})},\mbox{ for all }x\in S_{i-1;\delta}.\]
    \end{itemize}
\end{lemma}

\begin{proof}
    Let $K_{i;\delta}:=\ker(\pi_{i;\delta})$.
    First, we prove that {\rm (i)} implies {\rm (ii)}.
    For $i\in\{2,\ldots,\ell-1\}$, let $\sigma_i:K_{i;1}\rightarrow K_{i;2}$ 
    be the isomorphism satisfying Condition (\ref{eqn:addrestriction}) in Definition \ref{def:Comp}.
    For $i\in\{2,\ldots,\ell-1\}$, 
    let $\tau_{i;\delta}:S_{i-1;\delta}\rightarrow S_{i;\delta}$ be a transversal with respect to $\pi_{i;\delta}$. 
    For every $x\in S_{i-1;\delta}$, we have $\Inn(\tau_{i;\delta}(x)^{-1})^{K_{i;\delta}}\in\Inn(S_{i;\delta})^{K_{i;\delta}}$,
    therefore $(\sigma_i^{\bar{\delta}-\delta})_\bullet(\Inn(\tau_{i;\delta}(x)^{-1})^{K_{i;\delta}})\in \Aut(S_{i;{\bar{\delta}}})^{K_{i;{\bar{\delta}}}}$.
    This implies that for every $x\in S_{i-1;\delta}$,
    there exists $y\in\Aut(S_{i;\bar{\delta}})_{K_{i;\bar{\delta}}}$
    such that
    $(\sigma_{i}^{\bar{\delta}-\delta})_\bullet(\Inn(\tau_{i;\delta}(x)^{-1})^{K_{i;\delta}})=y^{K_{i;\bar{\delta}}}.$
    Let $\alpha_{i;\bar{\delta}}:S_{i-1;\delta}\rightarrow \Aut(S_{i;\bar{\delta}})^{K_{i;\bar{\delta}}}$
    be given by $\alpha_{i;\bar{\delta}}(x)=y$. 

    Now we show that {\rm (ii)} implies {\rm (i)}.
    Let $s\in S_{i;\delta}$. 
    It can be noted that $\{\tau_{i;\delta}(x)\mid x\in S_{i-1;\delta}\}$
    is a set of representatives of the coset space $[S_{i;\delta}:K_{i;\delta}]$. 
    Then so is the set $\{\tau_{i;\delta}(x)^{-1}\mid x\in S_{i-1;\delta}\}$.
    Therefore, $s=k\tau_{i;\delta}(x)^{-1}$
    for some $k\in K_{i;\delta}$ and $x\in S_{i-1;\delta}$.
    Hence
    \[\begin{aligned}
        (\sigma_i^{\bar{\delta}-\delta})_\bullet(\Inn(s)^{K_{i;\delta}})&=(\sigma_i^{\bar{\delta}-\delta})_\bullet(\Inn(k)^{K_{i;\delta}}\circ \Inn(\tau_{i;\delta}(x)^{-1})^{K_{i;\delta}})\\
        &=(\sigma_i^{\bar{\delta}-\delta})_\bullet(\Inn(k)^{K_{i;\delta}})\circ (\sigma_i^{\bar{\delta}-\delta})_\bullet(\Inn(\tau_{i;\delta}(x)^{-1})^{K_{i;\delta}})
    \end{aligned}
    \]
    Note that
    $(\sigma_i^{\bar{\delta}-\delta})_\bullet(\Inn(k)^{K_{i;\delta}})=\Inn(\sigma_i^{\bar{\delta}-\delta}(k))^{K_{i;\delta}}\in\Aut(S_{i;\bar{\delta}})^{K_{i;\bar{\delta}}}$
    and by our assumption, 
    $(\sigma_i^{\bar{\delta}-\delta})_\bullet(\Inn(\tau_{i;\delta}(x)^{-1})^{K_{i;\delta}})\in \Aut(S_{i;{\bar{\delta}}})^{K_{i;\bar{\delta}}}$.
    Therefore, we deduce that 
    \[(\sigma_i^{\bar{\delta}-\delta})_\bullet(\Inn(s)^{K_{i;\delta}})\in \Aut(S_{i;{\bar{\delta}}})^{K_{i;\bar{\delta}}}\]
    which completes the proof.
\end{proof}

\begin{lemma}\label{lem:Comp}
    For $\ell\ge 3$, the collection $\mathbf{Comp}_\ell$ is closed under $\#$ and $\mathbf{Comp}_{\ell}\searrow\mathbf{Comp}_{\ell-1}$. 
\end{lemma}

\begin{proof}
    Let $S_\delta:=((S_{i;\delta})_{0\le i\le\ell},(\pi_{i;\delta})_{1\le i\le\ell})$
    and $T_\delta:=((T_{i;\delta})_{0\le i\le\ell},(\rho_i)_{1\le i\le\ell})$ such that $S\approx T$.
    Let $(S_1,S_2),(T_1,T_2)\in \mathbf{Comp}_\ell$.
    We first show that $\mathbf{Comp}_\ell$ is closed under $\#$.
    Note that $S_\delta\# T_\delta\approx S_\delta$.
    Hence $(S_1\# T_1,S_2\# T_2)$ also satisfies Condition
    (\ref{eqn:addrestriction}) in Definition \ref{def:Comp}.
    Then by the definition of $\mathbf{Comp}_{\ell}$ for $\ell\ge 3$, 
    it is sufficient to show that $S_1\# T_1$ and $S_2\# T_2$ are compatible.
    Let $S'_{\ell;\delta}$ be the inverse limit of the inverse system $\mathcal{A}_\delta$ defined as
    \[\xymatrix@=0.5em{
        S_{\ell;\delta}\ar[rd]^{\pi_{\ell;\delta}}&\\
        &S_{\ell-1;\delta}\\
        T_{\ell;\delta}\ar[ru]_{\rho_{\ell;\delta}}
    }\]
    and let $\pi'_{\ell;\delta}:S'_{\ell;\delta}\rightarrow S_{\ell-1;\delta}$
    be the limit projection from $S'_{\ell;\delta}$ to $S_{\ell-1;\delta}$.
    By the definition of $\#$, 
    \[\xymatrix{
        S_\delta\# T_\delta=(S'_{\ell;\delta}\ar[r]^{\pi'_{\ell;\delta}}&S_{\ell-1;\delta}\ar[r]^{\pi_{\ell-1;\delta}}&\cdots\ar[r]^{\pi_{2;\delta}}&S_{1;\delta}\ar[r]^{\pi_{1;\delta}}&1).
    }\]
    Let $\mathcal{B}_\delta$ be the inverse system
    \[\xymatrix@=0.5em{
        S_{\ell-1;\delta}\ar@{=}[rd]\\
        &S_{\ell-1;\delta}\\
        S_{\ell-1;\delta}\ar@{=}[ru]
    }\] 
    Note that $S_{\ell-1;\delta}=\varprojlim(\mathcal{B}_\delta)$.
    By Theorem \ref{thm:projection} {\rm (ii)}, $\pi'_{\ell;\delta}$
    is the inverse limit of the morphism $\mathcal{A}_\delta\rightarrow\mathcal{B}_\delta$
    defined as
    \[\xymatrix@=0.5em{
        S_{\ell;\delta}\ar[rd]^{\pi_{\ell;\delta}}\ar[rrrr]^{\pi_{\ell;\delta}}&&&&S_{\ell-1;\delta}\ar@{=}[rd]\\
        &S_{\ell-1;\delta}\ar@{=}[rrrr]&&&&S_{\ell-1;\delta}\\
        T_{\ell;\delta}\ar[ru]^{\rho_{\ell;\delta}}\ar[rrrr]^{\rho_{\ell;\delta}}&&&&S_{\ell-1;\delta}\ar@{=}[ru]
    }\]
    Then by Theorem \ref{thm:limpullback} {\rm (iii)}, $\ker(\pi'_{\ell;\delta})$ is the inverse limit of the inverse system given by
    \[\xymatrix@=0.5em{
        \ker(\pi_{\ell;\delta})\ar[rd]&\\
        &1\\
        \ker(\rho_{\ell;\delta})\ar[ru]
    }\]
    Since $S_1$ and $S_2$ are compatible, we have that $\ker(\pi_{\ell;1})\cong\ker(\pi_{\ell;2})$.
    Similarly, we have that $\ker(\rho_{\ell;1})\cong\ker(\rho_{\ell;2})$.
    Hence
    \[\ker(\pi'_{\ell;1})\cong \ker(\pi_{\ell;1})\times\ker(\rho_{\ell;1})\cong\ker(\pi_{\ell;2})\times\ker(\rho_{\ell;2})\cong\ker(\pi'_{\ell;2})\]
    which implies that $S_1\# T_1$ and $S_2\# T_2$ are compatible. 

    Now we turn to the reducibility of $\mathbf{Comp}_\ell$ to $\mathbf{Comp}_{\ell-1}$.
    Let $(S_1,S_2)\in \mathbf{Comp}_\ell$ be reducible
    and for $2\le i\le\ell-1$, let $\sigma_i:\ker(\pi_{i;1})\rightarrow\ker(\pi_{i;2})$
    such that 
    \[(\sigma_i^{\bar{\delta}-\delta})_\bullet(\Inn(S_{i;\delta})^{\ker(\pi_{i;\delta})})\le \Aut(S_{i;\bar{\delta}})^{\ker(\pi_{i;\bar{\delta}})}\]
    where ${\bar{\delta}}=3-\delta$.
    Let $((S'_{i;\delta}),(\pi'_{i;\delta})):=\operatorname{Contra}(S_\delta)$
    and let $\ell':=\ell-1$.
    Note that for $2\le i\le \ell'-1$, $S'_{i;\delta}=S_{i;\delta}$ and 
    $\pi'_{i;\delta}=\pi_{i;\delta}$.
    We therefore let $\sigma'_i:=\sigma_i$ and we have 
    \[(\sigma'_i)^{\bar{\delta}-\delta}_\bullet(\Inn(S'_{i;\delta})^{\ker(\pi'_{i;\delta})})\le \Aut(S'_{i;\bar{\delta}})^{\ker(\pi'_{i;\bar{\delta}})}\]
    for $2\le i\le\ell'-1$, so $\operatorname{Contra}(S_1,S_2)\in \mathbf{Comp}_{\ell-1}$. 
    This completes the proof.
\end{proof}

\begin{lemma}
    \label{lem:almostequal}
    Let $\ell\ge 2$ and let $(S_1,S_2)$ and $(T_1,T_2)$ be pairs of compatible surjective sequences of length $\ell$. 
    If $(S_1,S_2)\in\mathbf{Comp}_\ell$ and $T_\delta\approx S_\delta$,
    then $(T_1,T_2)\in\mathbf{Comp}_\ell$.
\end{lemma}

\begin{proof}
    Write $S_\delta=((S_{i;\delta})_{0\le i\le\ell},(\pi_{i;\delta})_{1\le i\le\ell})$
    and $T_\delta=((T_{i;\delta})_{0\le i\ell},(\rho_{i;\delta})_{1\le i\le\ell-1})$.
    Since $(S_1,S_2)\in\mathbf{Comp}_\ell$, 
    there exists $\sigma_i$ for $2\le i\le\ell-1$ such that 
    \[(\sigma_i^{\bar{\delta}-\delta})_\bullet(\Inn(S_{i;\delta})^{\ker(\pi_{i;\delta})})\le \Aut(S_{i;\bar{\delta}})^{\ker(\pi_{i;\bar{\delta}})}.\]
    Note that $T_{i;\delta}=S_{i;\delta}$ and $\pi_{i;\delta}=\rho_{i;\delta}$ 
    for every $1\le i\le\ell-1$.
    We then have 
    \[(\sigma_i^{\bar{\delta}-\delta})_\bullet(\Inn(T_{i;\delta})^{\ker(\rho_{i;\delta})})\le \Aut(T_{i;\bar{\delta}})^{\ker(\rho_{i;\bar{\delta}})}.\]
    Hence $(T_1,T_2)\in\mathbf{Comp}_\ell$.
\end{proof}

\begin{theorem}\label{thm:comp_l}
    Let $\ell\ge 2$. If $S_\delta:=((S_{i;\delta})_{0\le i\le \ell},(\pi_{i;\delta})_{1\le i\le\ell})$
    are surjective group sequences of length $\ell$ and $(S_1,S_2)\in\mathbf{Comp}_\ell$,
    then there exists a good witness system associated with $(S_1,S_2)$. 
    In particular, $S_{\ell;1}$ and $S_{\ell;2}$ are compatible.
    
\end{theorem}

\begin{proof}
    We prove this by induction on $\ell$.
    For $\ell=2$, this follows from Lemma \ref{lem:length2}. Let $\bar{\delta}=3-\delta$.
    By  Definition~\ref{def:groupsequence+surjectivesequence}, all group sequences in $\mathbf{Comp}_\ell$ are surjective.
    For $\ell\ge 3$, by the induction hypothesis and Lemma \ref{lem:Comp},
    we only need to construct
    $G_\delta$ and $H_\delta$ with $\rho_\delta$, $\phi_\delta$ and $\eta_\delta$
    satisfying the conditions {\rm (ii)}, {\rm (iii)}, {\rm (iv)} in Lemma \ref{lem:induction3} with $\mathcal{C}_\ell=\mathbf{Comp}_\ell$.
    Following the notation in Lemma \ref{lem:comp_l}, for $2\le i\le\ell-1$,
    we let $\sigma_i:\ker(\pi_{i;1})\rightarrow\ker(\pi_{i;2})$ be isomorphisms,
    $\alpha_{i;\bar{\delta}}:S_{i-1;\delta}\rightarrow \Aut(S_{i;\bar{\delta}})_{\ker(\pi_{i;\bar{\delta}})}$ be maps,
    and $\tau_{i;\delta}:S_{i-1;\delta}\rightarrow S_{i;\delta}$ be transversals with respect to $\pi_{i;\delta}:S_{i;\delta}\rightarrow S_{i-1;\delta}$
    such that
    \[(\sigma_{i}^{\bar{\delta}-\delta})_\bullet(\Inn(\tau_{i;\delta}(x)^{-1})^{\ker(\pi_{i;\delta})})=\alpha_{i;\bar{\delta}}(x)^{\ker(\pi_{i;\bar{\delta}})},\mbox{ for all }x\in S_{i-1;\delta}.\]
    Let $n:=|S_{\ell-2;1}|$ ($=|S_{\ell-2;2}|$) and 
    $S_{\ell-2;\delta}:=\{x_{1;\delta},\ldots,x_{n;\delta}\}$ where $x_{1;\delta}:=1$.
    To be convenient, 
    we also let $\alpha_{\delta}:=\alpha_{\ell-1;\delta}$.
    Let $\mathcal{X}_\delta$ be the inverse system 
    \[\xymatrix@=0.5em{
        S_{\ell;\delta}\ar@{.}[dd]_n\ar[rd]^{\alpha_\delta(x_{1;\bar{\delta}})\circ\pi_{\ell;\delta}}&\\
        &S_{\ell-1;\delta}\\
        S_{\ell;\delta}\ar[ru]_{\alpha_\delta(x_{n;\bar{\delta}})\circ\pi_{\ell;\delta}}
    }\]
    and $G_\delta:=\varprojlim(\mathcal{X}_\delta)$.
    Let $\rho_\delta:G_\delta\rightarrow S_{\ell;\delta}$ 
    be the limit projection from
    $G_\delta$ to the first $S_{\ell;\delta}$ in the column shown in the above diagram.
    Let $T_\delta:=\ker(\pi_{\ell-1;\delta}\circ\pi_{\ell;\delta})$, let 
    $\iota_\delta:\ker(\pi_{\ell-1;\delta})\rightarrow S_{\ell-1;\delta}$ be the inclusion,
    and let $\theta_\delta:=\iota_\delta\circ\sigma^{\delta-\bar{\delta}}_{\ell-1}\circ (\pi_{\ell;\bar{\delta}}|_{T_{\bar{\delta}},\ker(\pi_{\ell-1;\bar{\delta}})})$.
    Define $H_\delta:=\operatorname{HW}(T_{\bar{\delta}},S_{\ell-1;\delta},\theta_\delta).$
    Let $\phi_\delta:H_\delta\rightarrow S_{\ell-1;\delta}$ 
    be the standard map of $H_\delta$.

    Now we compute $\ker(\pi_{\ell-1;\delta}\circ\phi_\delta)$
    and $\ker(\pi_{\ell-1;\delta}\circ\pi_{\ell;\delta}\circ\rho_\delta)$
    so that we can construct $\eta_\delta$. 
    Note that $\ker(\pi_{\ell-1;\delta})=\theta_\delta(T_{\bar{\delta}})$.
    Then
   $\ker(\pi_{\ell-1;\delta}\circ\phi_\delta)=\phi^{-1}_\delta(\ker(\pi_{\ell-1;\delta}))=\operatorname{BW}(T_{\bar{\delta}},S_{\ell-1;\delta},\theta_\delta)$.
    Let $\mathcal{Y}_\delta$ be the inverse system
    \[\xymatrix@=0.5em{
        T_{\bar{\delta}}\ar@{.}[dd]_{n}\ar[rd]^{\Inn(x_{1;\delta}^{-1})\circ\theta_\delta}\\
        &\ker(\pi_{\ell-1;\delta})\\
        T_{\bar{\delta}}\ar[ru]_{\Inn(x_{n;\delta}^{-1})\circ\theta_\delta}
    }\]
    By Lemma \ref{lem:bw} {\rm (iii)}, $\ker(\pi_{\ell-1;\delta}\circ\phi_\delta)=\varprojlim\mathcal{Y}_\delta$.
    As for $\ker(\pi_{\ell-1;\delta}\circ\pi_{\ell;\delta}\circ\rho_\delta)$,
    notice that $\pi_{\ell;\delta}\circ\rho_\delta$ is the inverse limit
    projection from $G_\delta$ to $S_{\ell-1;\delta}$.
    Let $\mathcal{S}_{\ell-1;\delta}$ be the inverse system
    \[\xymatrix@=0.5em{
        S_{\ell-1;\delta}\ar@{.}[dd]_n\ar@{=}[rd]&\\
        &S_{\ell-1;\delta}\\
        S_{\ell-1;\delta}\ar@{=}[ru]
    }\]
    Then by Theorem \ref{thm:projection} {\rm (ii)}, $\pi_{\ell;\delta}\circ\rho_\delta$
    is the inverse limit of the morphism $\mathcal{X}_\delta\rightarrow \mathcal{S}_{\ell-1;\delta}$
    defined as
    \[\xymatrix@=0.5em{
        S_{\ell;\delta}\ar@{.}[dd]_n\ar[rd]^{\alpha_\delta(x_{1;\bar{\delta}})\circ\pi_{\ell;\delta}}\ar[rrrrrr]^{\alpha_\delta(x_{1;\bar{\delta}})\circ\pi_{\ell;\delta}}&&&&&&S_{\ell-1;\delta}\ar@{=}[rd]\ar@{.}[dd]_(.3)n\\
        &S_{\ell-1;\delta}\ar@{=}[rrrrrr]&&&&& &S_{\ell-1;\delta}\\
        S_{\ell;\delta}\ar[ru]_{\alpha_\delta(x_{n;\bar{\delta}})\circ\pi_{\ell;\delta}}\ar[rrrrrr]_{\alpha_\delta(x_{n;\bar{\delta}})\circ\pi_{\ell;\delta}}&&&&&&S_{\ell-1;\delta}\ar@{=}[ru]
    }\]
    Let $\mathcal{S}_{\ell-2;\delta}$ be the inverse system
    \[\xymatrix@=0.5em{
        S_{\ell-2;\delta}\ar@{.}[dd]_n\ar@{=}[rd]&\\
        &S_{\ell-2;\delta}\\
        S_{\ell-2;\delta}\ar@{=}[ru]
    }\]
    Then $\pi_{\ell-1;\delta}$ is the inverse limit of the morphism $\mathcal{S}_{\ell-1;\delta}\rightarrow\mathcal{S}_{\ell-2;\delta}$
    defined as
    \[\xymatrix@=0.5em{
        S_{\ell-1;\delta}\ar@{.}[dd]_{n}\ar@{=}[rd]\ar[rrrrrr]^{\pi_{\ell-1;\delta}}&&&&&&S_{\ell-2;\delta}\ar@{.}[dd]_(.3)n\ar@{=}[rd]\\
        &S_{\ell-2;\delta}\ar[rrrrrr]^{\pi_{\ell-1;\delta}}&&&&& &S_{\ell-2;\delta}\\
        S_{\ell-1;\delta}\ar@{=}[ru]\ar[rrrrrr]^{\pi_{\ell-1;\delta}}&&&&&&S_{\ell-2;\delta}\ar@{=}[ru]
    }\]
    and $\pi_{\ell-1;\delta}\circ\pi_{\ell;\delta}\circ\rho_\delta$
    is the inverse limit of the morphism $\mathcal{X}_\delta\rightarrow \mathcal{S}_{\ell-2;\delta}$
    defined as
    \[\xymatrix@=0.5em{
        S_{\ell;\delta}\ar@{.}[dd]_n\ar[rd]^{\alpha_\delta(x_{1;\bar{\delta}})\circ\pi_{\ell;\delta}}\ar[rrrrrr]^{\pi_{\ell-1;\delta}\circ\alpha_\delta(x_{1;\bar{\delta}})\circ\pi_{\ell;\delta}}&&&&&&S_{\ell-2;\delta}\ar@{.}[dd]_(.3)n\ar@{=}[rd]\\
        &S_{\ell-1;\delta}\ar[rrrrrr]^{\pi_{\ell-1;\delta}}&&&&& &S_{\ell-1;\delta}\\
        S_{\ell;\delta}\ar[ru]_{\alpha_\delta(x_{n;\bar{\delta}})\circ\pi_{\ell;\delta}}\ar[rrrrrr]_{\pi_{\ell-1;\delta}\circ\alpha_\delta(x_{n;\bar{\delta}})\circ\pi_{\ell;\delta}}&&&&&&S_{\ell-2;\delta}\ar@{=}[ru]
    }\]
    Let $\mathcal{Z}_\delta$ be the inverse system
    \[\xymatrix@=0.5em{
        T_\delta\ar@{.}[dd]_n\ar[rd]^{\alpha_\delta(x_{1;\bar{\delta}})\circ\pi_{\ell;\delta}}&\\
        &\ker(\pi_{\ell-1;\delta})\\
        T_\delta\ar[ru]_{\alpha_\delta(x_{n;\bar{\delta}})\circ\pi_{\ell;\delta}}
    }\]
    By Theorem \ref{thm:limpullback} {\rm (iii)},   $\ker(\pi_{\ell-1;\delta}\circ\pi_{\ell;\delta}\circ \rho_\delta)=\varprojlim\mathcal{Z}_\delta$.
    
    Let $\eta_\delta:\ker(\pi_{\ell-1;\delta}\circ\phi_\delta)\rightarrow \ker(\pi_{\ell-1;\bar{\delta}}\circ\pi_{\ell;\bar{\delta}}\circ\rho_{\bar{\delta}})$ 
    be the inverse limit of the morphism $\mathcal{Y}_\delta\rightarrow \mathcal{Z}_{\bar\delta}$ defined as follows
    \[\xymatrix@=0.5em{
        T_{\bar\delta}\ar@{.}[dddd]_n\ar[rrdd]^{\Inn(x_{1;\delta}^{-1})\circ\theta_\delta}\ar@{=}[rrrrrr]&&&&&&T_{\bar\delta}\ar@{.}[dddd]^(.3)n\ar[rrdd]^{\alpha_{\bar{\delta}}(x_{1;\delta})\circ\pi_{\ell;{\bar{\delta}}}}\\
        &\\
        &&\ker(\pi_{\ell-1;\delta})\ar[rrrrrr]^{\sigma_{\ell-1}^{{\bar{\delta}}-\delta}\ \ \ \ \ }&&&&&&\ker(\pi_{\ell-1;{\bar{\delta}}})\\
        &\\
        T_{\bar{\delta}}\ar[rruu]_{\Inn(x_{n;\delta}^{-1})\circ\theta_\delta}\ar@{=}[rrrrrr]&&&&&&T_{\bar{\delta}}\ar[rruu]_{\alpha_{\bar{\delta}}(x_{n;\delta})\circ\pi_{\ell;{\bar{\delta}}}}
    }\]
    By Lemma \ref{lem:limiso} {\rm (ii)}, both $\eta_1$ and $\eta_2$ are isomorphisms.

    Now we prove that
    $G_\delta$, $H_\delta$, $\rho_\delta$, $\phi_\delta$ and $\eta_\delta$  
    satisfy Lemma \ref{lem:induction3} {\rm (ii)}, {\rm (iii)} and {\rm (iv)},
    which will complete the proof.
    Note that $\ker(\alpha_\delta(x_{1;{\bar{\delta}}})\circ \pi_{\ell;\delta})=\ker(\pi_{\ell;\delta})$.
    Part {\rm (ii)} holds by Lemma \ref{lem:goodproj}, where $\rho_\delta$ plays the role of $p_1$ and 
    $\mathcal{X}_\delta$ plays the role of $\mathcal{X}$.
    Using the notation in Lemma \ref{lem:induction3} {\rm (iv)},
    it can be seen that both $S'_\delta\approx S_\delta$ and $S''_\delta\approx S_\delta$.
    Note that 
    $(S_1,S_2)\in\mathbf{Comp}_\ell$.
    By Lemma \ref{lem:almostequal},
    $(S'_1,S'_2),(S''_1,S''_2)\in\mathbf{Comp}_\ell$. Hence part {\rm (iv)} holds.
    Now we turn to {\rm (iii)}.
    By Lemma \ref{lem:bw} , $\phi_\delta|_{\ker(\pi_{\ell-1;\delta}\circ\phi_\delta),\ker(\pi_{\ell-1;\delta})}$
    is the limit projection from $\ker(\pi_{\ell-1;\delta}\circ\phi_\delta)$ to $\ker(\pi_{\ell-1;\delta})$.
    Also by Theorem \ref{thm:limpullback} {\rm (ii)}, 
    $\pi_{\ell;{\bar{\delta}}}\circ\rho_{\bar{\delta}}|_{\ker(\pi_{\ell-1;\delta}\circ\pi_{\ell;\delta}\circ\rho_\delta),\ker(\pi_{\ell-1;\delta})}$
    is the limit projection from 
    $\ker(\pi_{\ell-1;\delta}\circ\pi_{\ell;\delta}\circ\rho_\delta)$
    to $\ker(\pi_{\ell-1;\delta})$.
    Then by the definition of $\eta_\delta$,
    we have the commutative diagram
    \[\xymatrix{
        \ker(\pi_{\ell-1;\delta}\circ\pi_{\ell;\delta}\circ\rho_\delta)\ar[rrr]^{\pi_{\ell;\delta}\circ\rho_\delta}\ar[d]^{\eta_\delta}&&&\ker(\pi_{\ell-1;\delta})\ar[d]^{\sigma^{{\bar{\delta}}-\delta}_{\ell-1}}\\
        \ker(\pi_{\ell-1;{\bar{\delta}}}\circ\phi_{\bar{\delta}})\ar[rrr]^{\phi_{\bar{\delta}}}&&&\ker(\pi_{\ell-1;{\bar{\delta}}})
    }\]
    This completes the proof.
\end{proof}

\begin{corollary}\label{cor:compnormalseries}
    Let $L_1$ and $L_2$ be groups with compatible normal series
    \[L_{0;1}:=1\trianglelefteq L_{1;1}\trianglelefteq \cdots\trianglelefteq L_{\ell;1}=:L_1\]
    and
    \[L_{0;2}:=1\trianglelefteq L_{1;2}\trianglelefteq \cdots\trianglelefteq L_{\ell;2}=:L_2.\]
    Let ${\bar{\delta}}:=3-\delta$.
    If for $2\le i\le \ell-1$, there exists an isomorphism $\sigma_i:L_{i;1}/L_{i-1;1}\rightarrow L_{i;2}/L_{i-1;2}$
    such that 
    \[(\sigma_i^{{\bar{\delta}}-\delta})_\bullet(\Inn(L_{\delta}/L_{i-1;\delta})^{L_{i;\delta}/L_{i-1;\delta}})\le \Aut(L_{\bar{\delta}}/L_{i-1;{\bar{\delta}}})^{L_{i;{\bar{\delta}}}/L_{i-1;{\bar{\delta}}}},\]
    then $L_1$ and $L_2$ are compatible.
\end{corollary}
\begin{proof}
    Let $S_{i;\delta}:=L_{\ell;\delta}/L_{\ell-i;\delta}$, let 
    $\pi_{i;\delta}:S_{i;\delta}\rightarrow S_{i-1;\delta}$ be given by 
    $xL_{\ell-i;\delta}\mapsto xL_{\ell-i+1;\delta}$ and let $S_\delta:=((S_{i;\delta}),(\pi_{i;\delta}))$.
    Since $L_{\ell-i;\delta}\le L_{\ell-i+1;\delta}$, we have that $\pi_{i;\delta}$ is well-defined.
    Observe that $\ker(\pi_{i;\delta})=L_{\ell-i+1;\delta}/L_{\ell-i;\delta}$.
    We have that $S_1$ and $S_2$ are compatible surjective group sequences.
    For $2\le i\le\ell-1$, 
    let $\sigma'_i:=\sigma_{\ell-i+1}:\ker(\pi_{i;1})\rightarrow \ker(\pi_{i;2})$.
    Then we have 
    \[(\sigma'_i)^{{\bar{\delta}}-\delta}_\bullet(\Inn(S_{i;\delta})^{\ker(\pi_{i;\delta})})\le \Aut(S_{i;{\bar{\delta}}})^{\ker(\pi_{i;{\bar{\delta}}})}.\]
    Therefore, $(S_1,S_2)\in\mathbf{Comp}_\ell$.
    By Theorem \ref{thm:comp_l}, this completes the proof.
\end{proof}

\begin{lemma}\label{lem:nilpotentcentral}
    Nilpotent groups of the same order have compatible central series.
\end{lemma}
\begin{proof}
    Let $L_1$ and $L_2$ be nilpotent groups of the same order.
    We use induction on $|L_1|$.
    Let $p$ be a prime such that $p\mid |L_1|$.
    It is known that if $G$ is nilpotent and $p\mid |G|$, then $p\mid |\operatorname{Z}(G)|$ (see \cite[Exercise 5.2.1]{robinson1996course}).
    Then by Cauchy's lemma, $L_{\delta}$ has a central subgroup $L_{1;\delta}$ of order $p$.
    Let $\pi_\delta:L_\delta\rightarrow L_\delta/L_{1;\delta}$ be the natural projection.
    By induction, there exist compatible central series,
    $1=M_{0;1}\trianglelefteq M_{1;1}\trianglelefteq \cdots\trianglelefteq M_{\ell;1}=L_1/L_{1;1}$
    and $1=M_{0;2}\trianglelefteq M_{1;2}\trianglelefteq \cdots\trianglelefteq M_{\ell;2}=L_2/L_{1;2}$.
    For every $i\in\{1,\ldots,\ell\}$, let $L_{i+1;\delta}$ be the full preimage of $M_{i;\delta}$ with respect to $\pi_\delta$.
    Therefore, $L_1$ and $L_2$ have central series:
    \[1=L_{0;1}\trianglelefteq L_{1;1}\trianglelefteq \cdots\trianglelefteq L_{\ell+1;1}=L_{1}\]
    and
    \[1=L_{0;2}\trianglelefteq L_{1;2}\trianglelefteq \cdots\trianglelefteq L_{\ell+1;2}=L_{2}.\]
    Note that $L_{1;1}\cong \mathbb{Z}_p\cong L_{1;2}$ 
    and for every $i\in\{2,\ldots,\ell+1\}$,
    \[L_{i;1}/L_{i-1;1}\cong M_{i-1;1}/M_{i-2;1}\cong M_{i-1;2}/M_{i-2;2}\cong L_{i;2}/L_{i-1;2}.\]
    Hence the central series above are compatible.
    This completes the proof.
\end{proof}

\begin{corollary}
    \label{cor:nilpotent}
    Let $L_1$ and $L_2$ be groups of the same order.
    If there exist $N_1\trianglelefteq L_1$ and $N_2\trianglelefteq L_2$
    such that $N_1\cong N_2$, and both $L_1/N_1$ and $L_2/N_2$ are nilpotent,
    then $L_1$ and $L_2$ are compatible. In particular, nilpotent groups of the same order are compatible.
\end{corollary}

\begin{proof}
    Let $L_{1;1}:=N_1$ and let $L_{1;2}:=N_2$.
    Since $L_1/N_1$ and $L_2/N_2$ are nilpotent groups of the same order, 
    then by Lemma \ref{lem:nilpotentcentral},
    there exist compatible central series,
    $1=M_{0;1}\trianglelefteq M_{1;1}\trianglelefteq \cdots\trianglelefteq M_{\ell;1}=L_1/N_1$
    and
    $1=M_{0;2}\trianglelefteq M_{1;2}\triangleleft \cdots \trianglelefteq M_{\ell;2}=L_2/N_2$.
    Let $\pi_\delta:L_\delta\rightarrow L_\delta/N_\delta$ be the natural projection.
    For every $i\in \{1,\ldots,n\}$, let $L_{i+1;\delta}$ be the full preimage
    of $M_{i;\delta}$ with respect to $\pi_\delta$.
    Therefore, there exist normal series of $L_1$ and $L_2$:
    \[1=L_{0;1}\trianglelefteq L_{1;1}\trianglelefteq \cdots\trianglelefteq L_{\ell+1;1}=:L_1\]
    and
    \[1=L_{0;2}\trianglelefteq L_{1;2}\trianglelefteq \cdots\trianglelefteq L_{\ell+1;2}=:L_2.\]
    Note that $L_{1;1}=N_1\cong N_2=L_{1;2}$ and for every $i\in\{2,\ldots,\ell+1\}$,
    \[L_{i;1}/L_{i-1;1}\cong M_{i-1;1}/M_{i-2;1}\cong M_{i-1;2}/M_{i-2;2}\cong L_{i;2}/L_{i-1;2}.\]
    Hence the normal series above are compatible.
    Also note that \[\Inn(L_\delta/L_{i-1;\delta})^{L_{i;\delta}/L_{i-1;\delta}}=1\]
    for $2\le i\le \ell$.
    By Corollary \ref{cor:compnormalseries}, we deduce that $L_1$ and $L_2$ are compatible.
\end{proof}

\begin{lemma}\label{lem:abelianautomorphism}
    Let $P$ and $Q$ be groups,
    let $A:=\Aut(P)$, let $\varphi:Q\rightarrow A$ be a homomorphism and let $G:=P\rtimes_\varphi Q$ be the semi-direct product of $P$ by $Q$.
    Then $\operatorname{C}_A(\varphi(Q))\le\Aut(G)^P$. 
    In particular, if $A$ is abelian, then $\Aut(G)^P=A$.
\end{lemma}
\begin{proof}
    Let $\sigma\in\operatorname{C}_A(\varphi(Q))$.
    We define a map $\tilde{\sigma}:G\rightarrow G$ by 
    \[\tilde{\sigma}((p,q)):=(\sigma(p),q).\]
    Let $(p_1,q_1),(p_2,q_2)\in G$.
    We have
    \[\begin{aligned}
        \tilde{\sigma}((p_1,q_1)(p_2,q_2))&=\tilde{\sigma}((p_1p_2^{q_1^{-1}},q_1q_2))\\
        &=(\sigma(p_1p_2^{q_1^{-1}}),q_1q_2)\\
        &=(\sigma(p_1)\sigma(p_2^{q_1^{-1}}),q_1q_2)
    \end{aligned}
    \]
    Since $\sigma\in\operatorname{C}_A(\varphi(Q))$, we have $\sigma(p_2^{q_1^{-1}})=(\sigma(p_2))^{q_1^{-1}}$.
    Therefore,
    \[\begin{aligned}
        \tilde{\sigma}((p_1,q_1)(p_2,q_2))&=(\sigma(p_1)\sigma(p_2)^{q_1^{-1}},q_1q_2)\\
        &=(\sigma(p_1),q_1)(\sigma(p_2),q_2)\\
        &=\tilde{\sigma}((p_1,q_1))\tilde{\sigma}((p_2,q_2)).
    \end{aligned}\]
    Hence $\tilde{\sigma}$ is a homomorphism.
    It is clear that $\tilde{\sigma}$ is bijective so $\tilde{\sigma}\in\Aut(G)$.
    Also note that $P$ is invariant under $\tilde{\sigma}$ and $\tilde{\sigma}^{P}=\sigma$.
    Therefore, we deduce that $\operatorname{C}_A(\varphi(Q))\le\Aut(G)^P$.
    In the case where $A$ is abelian, we have $\operatorname{C}_A(\varphi(Q))=A$.
    This implies that $A\le\Aut(G)^P$. It is clear that $\Aut(G)^P\le A$. Therefore $\Aut(G)^P=A$.
\end{proof}

\begin{corollary}\label{cor:abelianautomorphism}
    Let $P_1,\ldots,P_\ell$ be groups with abelian automorphism groups.
    Let $L_1$ and $L_2$ be groups of the same order.
    If there exist $N_1\trianglelefteq L_1$ and $N_2\trianglelefteq L_2$
    such that $N_1\cong N_2$
    and
    \[L_1/N_1\cong P_1\rtimes_{\phi_{1;1}}(P_2\rtimes_{\phi_{2;1}} (P_3\rtimes_{\phi_{3;1}}(\cdots(P_{\ell-1}\rtimes_{\phi_{\ell-1;1}} P_\ell))\cdots))\]
    and
    \[L_2/N_2\cong P_1\rtimes_{\phi_{1;2}}(P_2\rtimes_{\phi_{2;2}} (P_3\rtimes_{\phi_{3;2}}(\cdots(P_{\ell-1}\rtimes_{\phi_{\ell-1;2}} P_\ell))\cdots)),\]
    then $L_1$ and $L_2$ are compatible.
\end{corollary}

\begin{proof}
    Let $L_{1;1}:=N_1$ and let $L_{1;2}:=N_2$.
    Let $\pi_\delta:L_{\delta}\rightarrow L_\delta/L_{1;\delta}$.
    Let $L_{2;\delta}$ be the full preimage of $P_1$ with respect to $\pi_\delta$.
    For every $2\le i\le \ell+1$, 
    let $L_{i;\delta}$ be the full preimage of \[P_1\rtimes_{\phi_{1;\delta}}(P_2\rtimes_{\phi_{2;\delta}}(\cdots\rtimes(P_{i-1}\rtimes P_{i}))\cdots).\]
    Therefore, we have compatible normal series
    \[L_{0;1}:=1\trianglelefteq L_{1;1}\trianglelefteq \cdots\trianglelefteq L_{\ell+1;1}=:L_1\]
    and
    \[L_{0;2}:=1\trianglelefteq L_{1;2}\trianglelefteq \cdots\trianglelefteq L_{\ell+1;2}=:L_2.\]
    For $1\le i\le\ell+1$, let $\sigma_i:L_{i;1}/L_{i-1;1}\rightarrow L_{i;2}/L_{i-1;2}$
    be an isomorphism.
    Note that for every $2\le i\le\ell$,
    \[(\sigma_i^{{\bar{\delta}}-\delta})_\bullet(\Inn(L_\delta/L_{i-1;\delta})^{L_{i;\delta}/L_{i-1;\delta}})\le\Aut(L_{i;{\bar{\delta}}}/L_{i-1;{\bar{\delta}}}).\]
    By Lemma \ref{lem:abelianautomorphism}, 
    \[\Aut(L_{{\bar{\delta}}}/L_{i-1;{\bar{\delta}}})^{L_{i;{\bar{\delta}}}/L_{i-1;{\bar{\delta}}}}=\Aut(L_{i;{\bar{\delta}}}/L_{i-1;{\bar{\delta}}}).\]
    Then by Corollary \ref{cor:compnormalseries}, $L_1$ and $L_2$ are compatible.
\end{proof}

We are almost ready to prove that groups of the same square-free order are compatible. First, we need the following lemma about these groups. 
This lemma is well known (see, for example, \cite[10.1.10]{robinson1996course}).

\begin{lemma}\label{lem:square-free}
    Let $G$ be a group of square-free order.
    If $p$ is the largest prime dividing $|G|$, then $G$ has a normal Sylow p-subgroup.
\end{lemma}

\begin{corollary}\label{cor:square-free}
    Groups of the same square-free order are compatible with each other.
\end{corollary}
\begin{proof}
    Let $L_1$ and $L_2$ be groups such that
    $|L_1|=|L_2|=p_1p_2\cdots p_\ell$ where $p_1> p_2>\cdots> p_\ell$ are primes.
    Using Lemma \ref{lem:square-free}, the Schur-Zassenhaus Theorem and 
    applying induction,
    we get that both $L_1$ and $L_2$ are of form 
    $\mathbb{Z}_{p_1}\rtimes(\mathbb{Z}_{p_2}\rtimes(\cdots\rtimes(\mathbb{Z}_{p_{\ell-1}}\rtimes \mathbb{Z}_{p_{\ell}}))\cdots)$.
    Note that cyclic groups have abelian automorphism groups.
    Then by Corollary \ref{cor:abelianautomorphism}, $L_1$ and $L_2$ are compatible.
\end{proof}

To show the limits of our approach, we end the paper with an example of two groups that have compatible normal series but none of these series satisfy the hypothesis of Corollary \ref{cor:compnormalseries}.

\begin{example}\label{ex:counterexample}
The alternating group $A_5$ has a natural transitive action on a set $\Omega_1$ of size $5$. It also admits a transitive action on a set $\Omega_2$ of size $6$.
%    Let $A:=A_5$, let $\Omega_1:=[5]$ and let $\Omega_2:=[6]$. Then $A$ acts naturally on $\Omega_1$ and $A\cong\operatorname{PSL}(2,5)$ acts naturally on $\Omega_2$.
Let $L_1'=(A_5\wr_{\Omega_1} A_5)\times A_5$ and let $L_2'=A_5\wr_{\Omega_2} A_5$.
Note that $A_5$ has an irreducible representation of degree $4$ over $\mathbb{F}_7$, so we can view $A_5$ as an irreducible subgroup of $\GL(4,7)$. This induces a natural irreducible representation of degree $4\cdot 5$ of $A_5\wr_{\Omega_1}A_5$ which in turn induces a faithful representation $\varphi_1:L_1'\to \GL(4\cdot 6,7)$. Similarly, we also get a faithful irreducible representation $\varphi_2:L_2'\to \GL(4\cdot 6,7)$.
%Therefore, $A_5\wr_{\Omega_1}A_5\le\GL(4\cdot 5,7)$, $L_1'\le\GL(4\cdot 6,7)$ and $L_2'\le\GL(4\cdot 6,7)$.
%Let $\varphi_1$ and $\varphi_2$ be the corresponding representations of $L_1'$ and $L_2'$, respectively.
For $i\in\{1,2\}$, let $L_i:=\mathbb{Z}_7^{4\cdot 6}\rtimes_{\varphi_i}L_i'$.
Then $L_1$ and $L_2$ have a unique compatible normal series (with nontrivial factors), namely:
\[1\trianglelefteq \mathbb{Z}_7^{4\cdot 6}\trianglelefteq \mathbb{Z}_7^{4\cdot 6}\rtimes_{\varphi_1}A_5^6\trianglelefteq L_1\]
and
\[1\trianglelefteq \mathbb{Z}_7^{4\cdot 6}\trianglelefteq \mathbb{Z}_7^{4\cdot 6}\rtimes_{\varphi_2}A_5^6\trianglelefteq L_2.\]

Now, the hypothesis of Corollary~\ref{cor:compnormalseries} applied with $\ell=3$, $i=2$ and $\delta=2$ is that
\[\Inn(L_{2}/L_{1;2})^{L_{2;2}/L_{1;2}}\le \Aut(L_{1}/L_{1;1})^{L_{2;1}/L_{1;1}},\]
where the isomorphism $\sigma_2$ between the copies of $A_5^6$ in $L_1'$ and $L_2'$ is implicit. In other words, 
\[\Inn(L_2')^{A_5^6}\le \Aut(L_1')^{A_5^6}.\]
Note that $\Aut(L_1')=\Aut(A_5\wr_{\Omega_1} A_5)\times\Aut(A_5)$ so $\Aut(L_1')^{A_5^6}$ fixes one of the six $A_5$ factors of $A_5^6$.
On the other hand, $\Inn(L_2')^{A_5^6}$ is transitive on the six $A_5$ factors, so $\Inn(L_2')^{A_5^6}\nleq \Aut(L_1')^{A_5^6}$. Therefore, our methods cannot be applied to determine the compatibility of $L_1$ and $L_2$.
\end{example}

\section*{Acknowledgements}
This paper is based on the first author's PhD's thesis at the University of Auckland, which was financially supported by the Marsden Fund
(via grant 18-UOA-197) and the University of Auckland. 

\bibliographystyle{plain}
\bibliography{bib}

\end{document}